\documentclass[11pt]{amsart}

\usepackage{amsfonts}
\usepackage{color,enumerate,wrapfig,amssymb}
\usepackage{amsmath}
\usepackage{amsthm}
\usepackage{esint}
\usepackage{mathrsfs}
\usepackage{upgreek}
\usepackage{tipa}
\usepackage{tikz}
\usepackage{geometry}
\usepackage{hyperref}
\usepackage{eucal}

\theoremstyle{plain}
\newtheorem{theorem}{Theorem}
\newtheorem{lemma}[theorem]{Lemma}
\newtheorem{proposition}[theorem]{Proposition}
\newtheorem{corollary}[theorem]{Corollary}
\newtheorem{remark}[theorem]{Remark}

\newtheorem{definition}[theorem]{Definition}

\numberwithin{equation}{section}

\numberwithin{theorem}{section}

   \def\u{{\textbf u}}
      \def\h{{\textbf h}}
        
      \def\p{{\textbf p}}
            \def\q{{\textbf q}}
   \def\L{{\mathcal L}}
   
   \def\R{{\mathbb R}}
   \def\W{{\mathbb W}}
     \def\M{{\mathbb M}}
   \def\bH{{\mathbb H}}
        \def\e{{\textbf e}}
       \def\c{{\textbf c}}
           \def\v{{\textbf v}}    
                      \def\w{{\textbf w}}        
                      \def\h{{\textbf h}}       
  \def\kp{\kappa}                      
  \def\cH{ \mathcal{H}}

\def\XXint#1#2#3{{\setbox0=\hbox{$#1{#2#3}{\int}$ }
\vcenter{\hbox{$#2#3$ }}\kern-.6\wd0}}

\title[Regularity of the free boundary for a parabolic cooperative system]{Regularity of the free boundary \\ for a parabolic cooperative system}

\author{G. Aleksanyan}
\address{Department of Mathematics and Statistics, University of Helsinki, Finland. }
\author{ M. Fotouhi}
\address{Department of Mathematical Sciences, Sharif University of Technology, Tehran, Iran.}
\author{H. Shahgholian}
\address{Department of Mathematics, Royal Institute of Technology, 100~44  Stockholm, Sweden.}
\author{G. S. Weiss}
\address{Department of Mathematics, University of Duisburg-Essen, Essen, Germany.}

\begin{document}
\maketitle


\begin{abstract} 
In this paper we  study the following   parabolic system
  \begin{equation*}
  \Delta \u -\partial_t \u =|\u|^{q-1}\u\,\chi_{\{ |\u|>0 \}}, \qquad \u = (u^1, \cdots , u^m) \ ,
  \end{equation*}
  with free boundary $\partial \{|\u | >0\}$.   
 For $0\leq q<1$, we prove optimal growth rate for solutions $\u $ to the above system near free boundary points, and show   that  in  a uniform neighbourhood  of  any a priori well-behaved   free boundary  point   the free boundary is $C^{1, \alpha}$  in space directions and half-Lipschitz  
in the  time direction.
  \end{abstract}


\tableofcontents


\section{Introduction}


\subsection{Background}

In this paper we shall study for $0\leq q<1$ the parabolic (free boundary) system
\begin{equation}\label{system}
\begin{array}{l}
\Delta \u -\partial_t \u =f(\u):=|\u|^{q-1}\u\,\chi_{\left\lbrace |\u|>0\right\rbrace }, \qquad \u = (u^1, \cdots , u^m) \ ,
\end{array}
\end{equation}
where $\u: \ Q_1  \to \R^m$, 
where $Q_1 = B_1(0) \times (-1,1) $, with $B_1$ being the unit ball in $\R^n$,
$n\geq 2$, $m\geq 2$, and $|\cdot|$ is the Euclidean norm on the respective spaces. 
System \eqref{system} relates to  concentrations of species/reactants, where an increase in  each species/reactant
accelerates  the extinction/reaction of all species/reactants.
The special choice of our reaction kinetics would assure
a constant decay/reaction rate in the case that $u^i$, for $i=1, \cdots , m$ are
of comparable size.  

A diverse scalar parabolic free boundary problem has been subject of intense studies in more than half-century. On the other hand there are very few  results for problems that involve systems (see \cite{asuw15, CSY18, FSW20}), and probably no results for the system related to  equation \eqref{system}. 

The elliptic case of the above system is studied in  \cite{asuw15, FSW20} or in the scalar case, when $m=1$ in \cite{FS17}, where they prove optimal growth rate for the solutions as well as $C^{1,\alpha}$-regularity of the free boundary at  points that are a priori well-behaved.

In this paper we shall study the parabolic system \eqref{system} from a regularity point of view. The analysis of the above parabolic system introduces several serious obstruction, and hence a straightforward generalization of the ideas and techniques of its elliptic counterpart is far from being obvious. Due to its technical nature, and the need for notations and definitions,  we shall  explain these difficulties  below, during the course of developing the tools and ideas.


\subsection{Main results and plan of the paper}

Our results concern two main questions: Optimal growth  of the solution  
$\u$ at free boundary points (Theorem \ref{thm:growth}), and the regularity of the free boundary (Theorem \ref{thm:FBregularity})  at {\it well-behaved}  points.\footnote{Later  we shall  call them   regular  points.}

To prove our results we use  the  regularity theory for the elliptic case, see  \cite{asuw15, FSW20} and follow the ideas that have been used to treat parabolic free boundary problems,
 as in   \cite{CPS04} that  was used  for the  no-sign one phase scalar case. In doing so we encounter several technical problems, that we need to circumvent by enhancing the previous techniques.  
 The first problem we encounter  is the use of the balanced-energy monotonicity formula for proving quadratic growth estimates from the free boundary points. 
 In parabolic setting, and specially in system case, the combination of balanced energy and Almgren's frequency is more delicate than the elliptic case done in \cite{asuw15}. 

The second problem we encounter concerns  the regularity of the free boundary, where  we  are forced to use the  epiperimetric inequality in elliptic setting. In order to do this we need to prove that $\partial_t \u$,  the time derivative of $\u$, is H\"older regular for $q=0$.  When $q>0$ we need some modification (see Section 5). This, however,  can be  proved   at the so-called regular points. Indeed, since the set of regular points is open (in relative topology)   we can use indirect argument  to  show that $\partial_t \u$ tends to zero at free boundary points close to a regular point. From here one can bootstrap a H\"older regularity theory for $|\partial_t \u|$.  Once this is done we can invoke the epiperimetric inequality for equations with H\"older right hand side and deduce (in a standard way) the regularity of the free boundary in space. The H\"older regularity in time then follows by blow-up techniques, and indirect argument. 


\subsection{Notation}
For clarity of exposition we shall introduce some notation and definitions here  that are used frequently in the paper.\\

$\lfloor s\rfloor$ is the greatest integer below $s$, i.e. $s-1\leq\lfloor s\rfloor<s$.

Points in $\R^{n+1}$ are denoted by $(x,t)$, where $x \in \R^n$ and $t \in \R$.

Let $X=(x,t)$ and define $|X|:=(|x|^2+|t|)^{1/2}$.

$B_r(x)$ is the open ball in $\R^n$ with center $x$ and radius $r$, $B_r:=B_r(0)$.

$Q_r(x,t)$ denotes the open cylinder $B_r(x) \times (t-r^2,t+r^2)$ in $\R^{n+1}$.

$Q^+_r(x,t)=B_r(x) \times (t,t+r^2)$ (upper half cylinder).

$Q^-_r(x,t)=B_r(x) \times (t-r^2,t)$ (lower half cylinder).

$T_{r,a}:=B_a \times \left( -4r^2, -r^2 \right]$, $ T_r:=\mathbb{R}^n \times \left( -4r^2, -r^2 \right]$.
 
$\partial Q_r(x,t)$ is the topological  boundary.

$\partial_p Q_r(x,t)$ is the parabolic boundary, i.e., the topological boundary  minus the top of the cylinder.

$\nabla$ denotes the spatial gradient, $\nabla=(D_{x_1}, \cdots , D_{x_n})$.

$\nabla\u=[\partial_i u^j]_{1\leq i\leq n,1\leq j\leq m}$ is the derivative matrix of $\u$  with other notations
 \begin{align*}
|\nabla\u|^2&=\sum_{i=1}^m|\nabla u^i|^2,&\nabla\u:\nabla\v=\sum_{i=1}^m(\nabla u^i\cdot\nabla v^i),\\
\nabla\u\cdot\xi&=\xi^t\nabla\u=(\nabla u^1\cdot\xi,\cdots,\nabla u^m\cdot\xi),& \text{ for all }\xi\in\R^n.
\end{align*}

We will denote  the derivative of the function $f$ by $f_\u$.

We fix the following constants throughout the paper
\begin{equation}\label{def-alpha}
\kp:=\frac 2{1-q}, \qquad \alpha=(\kp(\kp-1))^{-\kp/2}.
\end{equation}

$\Gamma= \Gamma (\u) = \partial \{|\u|  >0\}.$

$\Gamma^\kp(\u)=\{(x_0,t_0)\in\Gamma(\u): \partial_t^i\partial_x^\mu\u(x_0,t_0)=0\text{ for all }2i+|\mu|<\kp\}$.

$\Omega_t$, $\Gamma_t$, $\partial \Omega_t$ are $t$-sections of
the corresponding sets in $\R^{n+1}$, at the level $t$.

$H= \Delta - \partial_t$ (the heat operator).

$\chi_\Omega$ is the characteristic function of $\Omega$.

We denote  by $G(x,s) $ the backward heat kernel 
\begin{equation*}
G(x,t)=\left\{\begin{array}{ll}(-4\pi t)^{-\frac{n}{2}}e^{\frac{|x|^2}{4t} }  , & t<0\\[10pt]
0,&t\geq0.\end{array}\right.
\end{equation*}

The following parabolic scalings at the point $X_0= (x_0,t_0) \in \Gamma$ are used,
\begin{equation*}
\u_{r,X_0}(x,t):= \frac{\u(rx+x_0,r^2 t+t_0)}{r^\kp},  \qquad 
 \u_r:=\u_{r,(0,0)}(x,t).
\end{equation*}
We say that $\u $ is $\kp$-backward self-similar if $\u_r=\u$ for all $r>0$, or equivalently $L u^i \equiv 0$, for $i=1,...,m$, where
\begin{equation*}
Lv:= \nabla v \cdot  x+2t\partial_t v -\kp v.
\end{equation*}

For  $\u $ a solution to the system \eqref{system} in $\R^n \times (-4,0 ]$, with a polynomial growth, we denote by $\mathbb{W}$ the parabolic balanced energy
\begin{equation} \label{weissrn}
\mathbb{W}(\u,r):=\frac{1}{r^{2\kp} }\int_{-4r^2}^{-r^2} \int_{\mathbb{R}^n}  \left(   | \nabla \u |^2 +\frac{ \kp|\u|^2}{2t } + \frac2{1+q} |\u |^{1+q} \right) G(x,t)dx dt,
\end{equation}
for $0<r<1$.
A change of variables implies that
\begin{equation*}
\mathbb{W}(\u, r)= \mathbb{W}(\u_r, 1).
\end{equation*}
For a fixed point $X_0=(x_0,t_0)\in \Gamma$, denote by 
\begin{equation*}
\W(\u, r; X_0):=\W(\u_{r,X_0},1).
\end{equation*}
For  notational simplicity we set 
\begin{equation*}
\M(\u):=\W(\u,1).
\end{equation*}

The class of half-space solutions $\mathbb H$ is defined as 
\begin{equation*}
\mathbb{H}:=\Big\{ x\mapsto \alpha\max (x \cdot \nu, 0)^\kp \e: \textrm{ where}, \nu \in \R^n, | \nu|=1, \e \in \R^m, |\e|=1 \Big\},
\end{equation*}
where $\alpha$ is defined in \eqref{def-alpha}.
A simple computation yields that $\W(\h,1)=:A_q$ is constant for every $\h\in\bH$.

We denote by $\mathcal{N}(r)$ the monotonicity function of Almgren
\begin{equation*}
\mathcal{N}(r)= \mathcal{N}(r,h):=\frac{ \int_{-4r^2}^{-r^2}\int_{\mathbb{R}^n} |\nabla h(x,t)|^2 G(x,t)dxdt}{\int_{-4r^2}^{-r^2}\int_{\mathbb{R}^n}\frac{1}{-t} |h(x,t)|^2 G(x,t)dxdt},
\end{equation*}
where $h$ is of polynomial growth in $x$-variables.


\section{Preliminary results and standard facts}\label{sec:preliminary}

\subsection{Monotonicity formulas}

In this section we shall present a few monotonicity formulas, that are the corner stone of our approach. The first of these is the standard balanced energy functional, that has strict monotonicity property for (global) solutions of our equation, unless the solution is backward self-similar of order $\kp$. See \cite{weiss1999self} for the similar result for the scaler case.

\begin{theorem}\label{monotonicity} (Monotonicity formula)
Let $\emph{ \u}$ be a solution of \eqref{system} in $ \R^n \times (-4,0)$, with a polynomial growth at infinity. Then $\mathbb{W}(\emph{\u},r)$ is monotone nondecreasing in $r$.
\end{theorem}
\begin{proof}
Using the identity
$$
\nabla v G=\nabla( v G)-\frac{x}{2t}vG,
$$
 we  compute the derivative of  $\W$ with respect to $r$
\begin{align*}
\frac{d\mathbb{W}(\u,r)}{dr}=&\frac{d\mathbb{W}(\u_r,1)}{dr}=
\int_{-4}^{-1} \int_{\mathbb{R}^n}  \frac{d}{dr}\left(   | \nabla \u_r |^2 + \frac{\kp|\u_r|^2}{2t } + \frac 2{1+q} |\u_r |^{1+q} \right) Gdx dt \\
=&2\int_{-4}^{-1} \int_{\mathbb{R}^n}  \left(   \nabla \u_r :\nabla \frac{d\u_r}{dr} + \frac{\kp\u_r}{2t } \frac{d\u_r}{dr}+\u_r|\u_r |^{q-1} \frac{d\u_r}{dr}\right) Gdx dt \\
=&2\int_{-4}^{-1} \int_{\mathbb{R}^n}\frac{d\u_r}{dr}  \left(   -\Delta \u_r +\u_r|\u_r |^{q-1}  
+ \frac{\kp\u_r}{2t }-\frac{x\cdot \nabla \u_r} {2t} \right) Gdx dt \\
=&\int_{-4}^{-1} \int_{\mathbb{R}^n}\frac{d\u_r}{dr}  \left(  - 2\frac{\partial \u_r} {\partial t} 
 + \kp\frac{\u_r}{t }-\frac{x\cdot \nabla \u_r} {t} \right) Gdx dt \\
 =&r\int_{-4}^{-1}\int_{\mathbb{R}^n}\left(\frac{d\u_r}{dr}  \right)^2 \frac{G(x,t)}{-t}dxdt 
 \geq 0.
\end{align*}
\end{proof}

The above monotonicity functional being limited to global solutions, needs to be enhanced 
in order for us to apply to a local setting. This is done by inserting a cutoff function into the functional, that in turn makes the functional almost monotone and calls for adding an extra term, as stated in the next theorem. See also \cite{CPS04} for the similar result in obstacle problem.

\begin{theorem}\label{cutoff}
Given a solution $\emph{\u}$ to \eqref{system} in  $ Q_1^-$, we consider the function $\emph{\v}:=\eta \emph{\u} $, where $ \eta \in C_0^\infty(B_{3/4})$ is nonnegative, 
$\eta \leq 1$, 
and $\eta =1$ in $ B_{1/2}$. Then there exists  a non-negative  function
 $F$ depending  on the given data, satisfying $ F( 0+) = 0$, and such that  $ \mathbb{W}(\emph{\v},r)+F(r)$ is monotone nondecreasing in $r$ for $ 0<r<1/2$.
\end{theorem}
\begin{proof}
As in the previous theorem and applying the relation $\frac{d\v_r}{dr}=\frac1rL\v_r$, we get
\begin{equation*}
\begin{aligned}
\frac{d\mathbb{W}(\v,r)}{dr}=\frac{d\mathbb{W}(\v_r,1)}{dr}=&
\frac{2}{r^{2\kp+1}}\int_{-4r^2}^{-r^2} \int_{\mathbb{R}^n}L\v \left(   -\Delta \v +\v|\v |^{q-1} 
+ \frac{\kp\v}{2t }-\frac{x\cdot \nabla \v} {2t} \right) Gdx dt \\
=&\frac{2}{r^{2\kp+1}}\int_{-4r^2}^{-r^2} \int_{\mathbb{R}^n}L\v \left(   -\Delta \v +\v|\v |^{q-1}  +\partial_t \v
+ \frac{L\v}{-2t } \right) Gdx dt \\
\geq& \frac{2}{r^{2\kp+1}}\int_{-4r^2}^{-r^2} \int_{\mathbb{R}^n}L\v \left(   -\Delta \v +\v|\v |^{q-1}  +\partial_t\v \right) Gdx dt ,
\end{aligned}
\end{equation*}
Observe that  $ H\v= H\u =\u|\u |^{q-1}=\v|\v |^{q-1}$ in $B_{1/2}$, and $H\v(x,t)=0$  if $ |x|>3/4$, hence
\begin{equation*}
\begin{aligned}
\frac{d\mathbb{W}(\v,r)}{dr}
\geq& \frac{2}{r^{2\kp+1}}\int_{-4r^2}^{-r^2} \int_{\mathbb{R}^n}L\v \left(   -H\v +\v|\v |^{q-1} \right) Gdx dt \\
=&
\frac{2}{r^{2\kp+1}}\int_{-4r^2}^{-r^2} \int_{B_{3/4}\setminus B_{1/2}}L\v \left(  -H \v +\v|\v |^{q-1} \right) Gdx dt
\geq -\frac{Ce^{\frac{-1}{64r^2}}}{r^{n+2\kp-1}},
\end{aligned}
\end{equation*}
where we have used the following relations
\begin{equation*}
L \v= (x\cdot \nabla \eta) \u+\eta L\u, \textrm{ and } H\v= \Delta \eta  \u+\eta H\u+2 \nabla \u \cdot\nabla \eta.
\end{equation*}
Now the statement of the lemma follows with 
\begin{equation*}
F(r)=C\int_0^r \tau^{-n-2\kp+1}e^{\frac{-1}{64 \tau^2}}d\tau.
\qedhere
\end{equation*}
\end{proof}

We state the following standard result  concerning  regularity theory, leaving out the standard proof. See for example \cite{W2000} for similar result. 
\begin{corollary}
Let $\emph{\u} $ be a solution to our problem and suppose it has polynomial growth from a free boundary point $X_0=(x_0,t_0)$, where both $\emph{\u}$ and $\nabla\emph{\u}$ vanish. Then the following hold.
\begin{enumerate}
  \item The function $ \mathbb{W}(\emph{\u},r;X_0)$ has a right limit as $ r\rightarrow 0+$.
  \item Any blow up of $ \emph{\u}$ at $ (x_0,t_0)$ is a $\kp$-backward self-similar function.
   \item The function $ X_0\mapsto \mathbb{W}(\emph{\u}, 0+;X_0)$ is upper semicontinous.
\end{enumerate}
\end{corollary}

\bigskip

Next we state, and for reader's convenience, prove Almgren's monotonicity formula. There are different versions of this formula in literature, see for example \cite{danielli2017optimal}.

\begin{lemma} (Almgren's frequency formula)  \label{Nr}
Let $h$ be a non-zero caloric function in  $\R^n\times(-4,0)$, with polynomial growth, and 
recall the definition of Almgren's monotonicity function $ \mathcal{N} (r,h)$. 
Then 
\begin{itemize}
\item[i)]
$ \mathcal{N}^{\prime}(r,h) \geq 0$, for $0< r <1$. 
\item[ii)]
If $ \mathcal{N}(r,h) \equiv const :=\mathcal{N} $, then $ h$ is a backward self-similar caloric function of degree $ 2\mathcal{N}$. 
\item[iii)]
For an integer number $\ell\geq2$, if $\partial_t^j\partial_x^\mu h (0) = 0$ 
 for all $2j+|\mu|\leq\ell-1$, we obtain $2\mathcal{N}(0+,h) \geq  \ell$ and equality implies that $h$ is backward self-similar of degree $\ell$.
\end{itemize}
\end{lemma}

\begin{proof}
We have
\begin{equation*}
\mathcal{N}(r):=\frac{\int_{-4}^{-1} \int_{\mathbb{R}^n} |\nabla h_r|^2 Gdxdt}{\int_{-4}^{-1} \int_{\mathbb{R}^n} \frac1{-t}|h_r|^2 Gdxdt},
\end{equation*}
and
\begin{equation*}
\begin{aligned}
 \mathcal{N}^{\prime}(r)= \frac{ 2 I_1 \int_{-4}^{-1} \int_{\mathbb{R}^n}\frac1{-t} |h_r|^2 Gdxdt- 2 I_2\int_{-4}^{-1} \int_{\mathbb{R}^n}\frac1{-t} h_r\frac{dh_r}{dr}Gdxdt}{\left( \int_{-4}^{-1} \int_{\mathbb{R}^n}\frac1{-t} |h_r|^2 Gdxdt \right)^2},
\end{aligned}
\end{equation*}
where
\begin{equation*}
I_1:= \int_{-4}^{-1}\int_{\mathbb{R}^n} \nabla h_r \cdot \nabla \frac{d h_r}{dr} Gdxdt ~~ \textrm{ and } ~
I_2:= \int_{-4}^{-1} \int_{\mathbb{R}^n} | \nabla h_r|^2 Gdxdt.
 \end{equation*}

\par
Let us recall that $ \frac{d h_r}{dr}= \frac{1}{r}L h_r$. Using
integration by parts, and taking into account that $h$ is caloric, we obtain
\begin{equation*}
I_1:=  \int_{-4}^{-1} \int_{\mathbb{R}^n} \left(-\Delta h_r  - \frac{x \cdot \nabla h_r}{2t}\right)\frac{dh_r}{dr}  Gdxdt
 =   \int_{-4}^{-1}\int_{\mathbb{R}^n} \frac{-1}{2rt}\left((Lh_r)^2 +\kp h_r Lh_r\right)Gdxdt.
\end{equation*}
By similar computations,
\begin{equation}\label{eq:I_2}
I_2:= 
\int_{-4}^{-1} \int_{\mathbb{R}^n} \left(-\Delta h_r  - \frac{x \cdot \nabla h_r}{2t}\right)h_r  Gdxdt
 =  \int_{-4}^{-1}\int_{\mathbb{R}^n}\frac{-1}{2t} (Lh_r +\kp h_r)h_r Gdxdt.
  \end{equation}
Now consider
\begin{equation*}
\begin{aligned}
r \mathcal{N}^{\prime}(r) \left( \int_{-4}^{-1} \int_{\mathbb{R}^n} \frac{-1}{t}|h_r|^2 Gdxdt \right)^2= &
\left( \int_{-4}^{-1}\int_{\mathbb{R}^n} \frac{-1}{t}(Lh_r)^2Gdxdt\right)\left(\int_{-4}^{-1} \int_{\mathbb{R}^n}\frac{-1}{t} |h_r|^2 Gdxdt\right) \\
&\qquad-\left(\int_{-4}^{-1} \int_{\mathbb{R}^n}\frac{-1}{t} h_rLh_r Gdxdt\right)^2\geq0.
  \end{aligned}
   \end{equation*}
Hence  $\mathcal{N}$ is nondecreasing, and if $  \mathcal{N}^{\prime} =0$, then $ Lh_r=ch_r $. Recalling that $ Lh_r= x \cdot \nabla h_r +2t\partial_t h_r-\kp h_r $, we obtain 
$  x \cdot \nabla h -2t\partial_t h-(\kp+c)h =0$, which  is equivalent to $ h $
 being backward self-similar of degree $ c+\kp$.
 On the other hand, we have from \eqref{eq:I_2}
 \begin{equation*} 
\mathcal{N}(r)=\frac{\int_{-4}^{-1} \int_{\mathbb{R}^n} |\nabla h_r|^2 Gdxdt}{\int_{-4}^{-1} \int_{\mathbb{R}^n} \frac1{-t}|h_r|^2 Gdxdt} =\frac{c+\kp}{2} ,
\end{equation*} 
hence $ c+\kp=2 \mathcal{N}$. 

The last statement of the lemma follows now by  the contradiction argument. Suppose that $2 \mathcal{N}(s)<\ell$ for some $s\in(0,1]$, it follows that $2 \mathcal{N}(0+)<\ell$. By scaling 
$$w_r:=\frac{h_r}{\left(\int_{-4}^{-1} \int_{\mathbb{R}^n} \frac1{-t}|h_r|^2 Gdxdt\right)^{1/2}},$$
we infer from the boundedness of $ \mathcal{N}(r)$ that $\{w_r\}$ is bounded\footnote{ 
Note that $G\geq \frac{e^{-R^2/4}}{(16\pi)^{n/2}}$ for  $|x|\leq R$, and $h$ is of polynomial growth.}
 in $L^2(-4,-1;W^{1,2}(B_R))$ for every $R>0$.
Now we apply Lemma \ref{caloric-estimate-lemma} for $w_r$ and $\nabla w_r$ to get that $\{w_r\}$ is bounded in $L^2(-4,0;W^{1,2}(B_R))$. Indeed, for $-4<s<-2$ and $-1<t<0$ we can write
\begin{align*}
\int_{-1}^0\int_{B_R}|w_r(x,t)|^2dxdt\leq &\int_{-1}^0\int_{B_R}\int_{-4}^{-2}\int_{\R^n}\frac12e^{-\frac{|x|^2}{t+s}}\left(\frac{\sqrt3s}{s-t}\right)^{n}|w_r(y,s)|^2G(y,s)dydsdtdx\\
\leq& e^{R^2/2}2^{2n+1}3^{n/2}|B_R|\int_{-4}^{-2}\int_{\R^n}\frac1{-s}|w_r(y,s)|^2G(y,s)dyds\\
\leq & e^{R^2/2}2^{2n+1}3^{n/2}|B_R|.
\end{align*}

Furthermore, the estimates on derivatives for caloric functions imply that $\{w_r\}$ is bounded in $L^2(-3,0;W^{2,2}(B_R))$. 
Consequently, by diagonalization technique there is a weakly convergence sequence $w_{r_m}\rightharpoonup w_0$  in $W^{1,2}(-3,0;W^{2,2}_{{\rm loc}}(\R^n))$ as well as $w_{r_m}\rightarrow w_0$ strongly in $L^2(-3,0;W^{1,2}_{{\rm loc}}(\R^n))$. 
Therefore, the limit $w_0$ is a caloric function satisfying $w_0(0)=\partial_t^j\partial_x^\mu w_0 (0) = 0$ for all $2j+|\mu|\leq\ell-1$. The later inequality  is a consequence of $w_{r_m}$ being smooth and their derivatives being  uniformly bounded by  $\lVert w_{r_m}\rVert_{L^1\left((-4,0)\times B_R\right)}$.
We claim  now that for every fixed $0<r\leq1/3$,
\begin{equation}\label{almgren:normw0}
\int_{-4r^2}^{-r^2} \int_{\mathbb{R}^n} \frac1{-t}|w_0|^2 Gdxdt=\lim_{r_m\rightarrow0}\int_{-4r^2}^{-r^2} \int_{\mathbb{R}^n} \frac1{-t}|w_{r_m}|^2 Gdxdt=1,
\end{equation}
and
\begin{equation}\label{almgren:claim}
\int_{-4r^2}^{-r^2} \int_{\mathbb{R}^n}|\nabla w_0|^2 Gdxdt=\lim_{r_m\rightarrow0}\int_{-4r^2}^{-r^2} \int_{\mathbb{R}^n} |\nabla w_{r_m}|^2 Gdxdt.
\end{equation}
Suppose this is true, then we for $r<1/3$ we have 
$$\mathcal{N}(r,w_0)= \lim_{r_m\rightarrow0}\mathcal{N}(r,w_{r_m})=\lim_{r_m\rightarrow0}\mathcal{N}(rr_m,h)=\mathcal{N}(0+,h).$$
So, $w_0$ must be a backward self-similar  function of degree $2\mathcal{N}(0+,h)<\ell$ for $0<t<1$. Since  $w_0$ is caloric function, so $2\mathcal{N}(0+,h)\in\mathbb{N}$, comparing with  $w_0(0)=\partial_t^j\partial_x^\mu w_0 (0) = 0$ for all $2j+|\mu|\leq\ell-1$, this yields a contradiction with \eqref{almgren:normw0}.

Therefore, $2\mathcal{N}(s,h)\geq\ell$ for $s\in(0,1]$. If   $2\mathcal{N}(1,h)=\ell$, then $\mathcal{N}$ is constant on $(0,1)$ and thereby $h$ is a backward self-similar function of degree $\ell$.

To close the argument, we need to prove \eqref{almgren:normw0} and \eqref{almgren:claim}. This is a matter of computation and can be settled easily by Lemma \ref{caloric-estimate-lemma}.
Indeed, we just need to show the following uniform convergence when $0<r\leq1/3$ is fixed and $r_m\rightarrow 0$,
\begin{align*}
\int_{-4r^2}^{-r^2} \int_{\R^n\setminus B_R} \frac1{-t}|w_{r_m}|^2 Gdxdt\leq &\left(\int_{-4r^2}^{-r^2} \int_{\R^n\setminus B_R}(2\sqrt3)^{n}e^{\frac{|x|^2}{3-t}}\frac1{-t}G(x,t)dxdt\right)\\
&\times\left(\int_{-4}^{-3}\int_{\R^n}\frac1{-s}|w_{r_m}(y,s)|^2G(y,s)dyds\right)\\
\leq&\int_{-4r^2}^{-r^2} \int_{\R^n\setminus B_R}\frac\pi3\left(\frac{3}{-\pi t}\right)^{n/2+1}\exp\left(\frac{|x|^2}{3-t}+\frac{|x|^2}{4t}\right)dxdt\\
\leq&\int_{-4r^2}^{-r^2} \int_{\R^n\setminus B_R}\frac\pi3\left(\frac{3}{-\pi t}\right)^{n/2+1}e^{\frac{|x|^2}{8t}}dxdt\rightarrow 0 \text{\ \  as }R\rightarrow \infty.
\end{align*}
The proof of \eqref{almgren:claim} is the same if we apply again Lemma  \ref{caloric-estimate-lemma} for the caloric function $\nabla w_{r_m}$.
\end{proof}


\subsection{Nondegeneracy}

\begin{proposition} {(Nondegeneracy)}\label{Nondegeneracy}
Let $\emph{\u}$ be a solution of \eqref{system} with $0\leq q<1$. Then there is a positive constant $c=c(q,n)$ such that 
if $ (x_0,t_0) \in \overline{\{ |\emph{\u}| > 0\}}$, and $ Q^-_r(x_0,t_0) \subset Q_1$,  then 
\begin{equation*}
\sup_{Q^-_r (x_0,t_0)}|\emph{\u}| \geq c r^\kp.
\end{equation*}
\end{proposition}

\begin{proof}
Let $ U(x,t):= |\u(x,t)|^{1-q}$. The proof follows in a standard way using
\begin{equation*}
\Delta U -\partial_t U = (1-q)+(1-q)\frac{ | \nabla \u |^2 }{U^{\kp-1}}-\frac{1+q}{1-q}\frac{|\nabla U|^2}{U}, \qquad\text{ in } \{U>0\}.
\end{equation*}
For any $( y, s) \in \{ | \u | > 0\} $, (close to $(x_0,t_0)$), set $w(x,t)=c(|x-y|^2+(s-t))$ for small constant $c>0$ to be specified later. 
Then $h=U-w$ satisfies in $ Q_r^-(y,s)$
\begin{align*}
\mathcal Lh-\partial_t h&:=\Delta h-\partial_t h+\frac{1+q}{1-q}\left(\frac{\nabla(U+w)}U\cdot\nabla h-\frac{4c}Uh\right)\\
&=(1-q)-4c(\frac{2n+1}4+\frac{1+q}{1-q})+(1-q)\frac{|\nabla\u|^2}{U^{\kp-1}}+4c^2\frac{1+q}{1-q}\frac{s-t}U\geq0,
\end{align*}
provided that $c$ is small enough. In particular $h$ cannot attain a local maximum in $Q_r^-(y,s)\cap \{ | \u | > 0\}$ according to the maximum principle for $\mathcal L-\partial_t$. 
On the other hand $h<0$ on $\partial\{|\u|>0\}$ and hence the positive maximum of $h$ is attained on  $ \partial_pQ_r^-(y,s)$, and we conclude that
\[
\sup_{\partial_pQ_r^-(y,s)}(U-w)\geq U(y,s)>0,
\]
which amounts to 
\[
\sup_{\partial_pQ_r^-(y,s)}U\geq cr^2.
\]
Letting $(y,s)\rightarrow(x_0,t_0)$, we arrive at the statement of the proposition.
\end{proof}


\section{Regularity of solutions}\label{sec:quadratic}
In this section we   study the regularity of solutions (to  equation \eqref{system}), which  
according to the parabolic  regularity theory,  are  known  to be  $C_x^{1,\beta}\cap C_t^{0,(1+\beta)/2}$ for $q=0$ and $C_x^{2,\beta}\cap C_t^{1,\beta/2}$ for $q>0$.  
Here we will show the optimal  growth for solutions from points where $u$ vanish to the highest order for our problem.
In order to study the optimal growth (regularity) of solution, we start with the following definition; see also \cite{soave2018nodal}.
\begin{definition}
The vanishing order of $\emph{\u}$ at point $X_0$ is defined to be  the largest value $\mathcal V(X_0)$ which satisfies 
$$
\limsup_{r\rightarrow0^+}\frac{\lVert\emph{\u}\rVert_{L^\infty(Q_r^-(X_0))}}{r^{\mathcal V(X_0)}}<+\infty.
$$
\end{definition}
One of main tools in studying a sublinear equation is Lemma \ref{holder-regularity}, which is the dual of 
 \cite[Lemma 1.1]{caffarelli1985partial} for the elliptic case. For the convenience of reader we put the proof in the Appendix. 
One of the useful result of Lemma \ref{holder-regularity} is that if $\u$ is a solution of \eqref{system} and $X_0\in\Gamma(\u)$, then 
$\mathcal V(X_0)\in\{1,2,3,\cdots,\lfloor\kp\rfloor,\kp\}$. 
Moreover, we can find out easily that if $\mathcal V(X_0)=s<\kp$, then $\partial_t^i\partial_x^\mu\u(X_0)$ exists and vanishes for $2i+|\mu|<s$. Indeed, there is a self-similar vectorial polynomial $P$ of degree $s$ such that $|\u(X)-P(X)|\leq C|X-X_0|^s$.
Our main result for case $q>0$ is that if $X_0\in \Gamma^\kp(\u)=\{X_0\in\Gamma(\u): \partial_t^i\partial_x^\mu\u(X_0)=0\text{ for all }2i+|\mu|<\kp\}$, then $\mathcal V(X_0)=\kp$.

We start with the following lemma which is essential to obtain our result.

\begin{lemma}\label{estimate-solution}
For any  $ \emph{\u}$ solving  \eqref{system} in $Q_2$, and  satisfying the doubling 
\begin{equation}\label{doubling}
\lVert\emph{\u}\rVert_{L^\infty(Q_2^-)}\leq   2^\kp \lVert\emph{\u}\rVert_{L^\infty(Q_{1}^-)},
\end{equation}
we have 
$$
\lVert\emph{\u}\rVert_{L^\infty(Q_{1}^-)}\leq \max\left\{1,C\lVert \emph{\u}G^{1/(1+q)}\rVert_{L^{1+q}(Q_1^-)} \right\},
$$
where  $C$ is independent of  $\emph{\u}$.
\end{lemma}

\begin{proof}
Suppose the statement of the lemma fails. Then there is a sequence $\u_j$ satisfying the hypothesis of the lemma with \begin{equation}\label{contra1}
\lVert\u_j\rVert_{L^\infty(Q_{1}^-)}\geq1,\qquad\text{and}\qquad\lVert\u_j\rVert_{L^\infty(Q_{1}^-)}\geq j \lVert \u_jG^{1/(1+q)}\rVert_{L^{1+q}(Q_1^-)} .
\end{equation}
Define $\tilde \u_j = \u_j/ \lVert\u_j\rVert_{L^\infty(Q_{1}^-)}$, and insert in \eqref{contra1}, to arrive at 
\begin{equation}\label{contra2}
\frac1j \geq \lVert \tilde \u _jG^{1/(1+q)}\rVert_{L^{1+q}(Q_1^-)} .
\end{equation}
Since   $\tilde \u_j$ satisfies the doubling \eqref{doubling}, then it yields 
$$ \lVert H (\tilde \u_j)\rVert_{L^\infty(Q_{2}^-)}\leq2^{\kp q} \lVert \u_j\rVert_{L^\infty(Q_{1}^-)}^{q-1}\leq 2^{\kp q}.$$
Therefore we have a subsequence of $\tilde \u_j$ which converges to a limit function $\u_0$ satisfying 
$$
\lVert \tilde{\u}_0G^{1/(1+q)}\rVert_{L^{1+q}(Q_1^-)}= 0, \qquad \lVert\tilde \u_0\rVert_{L^\infty(Q_{1}^-)} = 1,
$$
which is obviously a contradiction.
\end{proof}

\begin{theorem}\label{thm:growth}
For $\emph{\u}$ a   solution to \eqref{system}, with  $(0,0)\in\Gamma^\kp(\emph{\u})$,  there exists a constant $C$ such that
\begin{equation*}
\sup_{Q^-_r} |\emph{\u}| \leq C r^{\kp}, \qquad  \forall \ 0<r<1/2 .
\end{equation*}
\end{theorem}

\begin{proof}
{\bf Case $\kp\notin\mathbb{N}$:} The proof in this case follows by standard blow-up and the use of Liouville's theorem, and the only subtle point would be to prove the blow-up solution will vanish at the origin, of order $\kappa$; the latter is taken care of in  Appendix. here is how it works out. If the statement of the theorem fails, then there exists a sequence $r_j\rightarrow0$ such that
$$
\sup_{Q_r^-}|\u|\leq jr^\kp,\qquad \forall r\geq r_j,\qquad \sup_{Q_{r_j}^-}|\u|=jr_j^\kp.
$$
In particular the function $\tilde\u_j(x,t)=\frac{\u(r_jx,r_j^2t)}{jr_j^\kp}$ satisfies
$$\sup_{Q_R^-}|\tilde\u_j|\leq R^\kp,\qquad\text{ for }1\leq R\leq\frac1{r_j},$$
with equality for $R=1$, along with
$$H\tilde\u_j=\frac{f(\tilde\u_j)}{j^{1-q}}\longrightarrow 0\ \text{ uniformly in }Q_R^-.$$
From this we conclude that there is a convergent subsequence, tending to a caloric function $\u_0$ with growth $\kp$, i.e.
\begin{equation}\label{u0-property}
\sup_{Q_R^-}|\u_0|\leq R^\kp,\quad\forall R\geq1,\quad\sup_{Q_1^-}|\u_0|=1,\quad H\u_0=0,
\end{equation}
and furthermore, $|\tilde\u_j(X)|\leq C_0|X|$ in $Q_1^-$ uniformly for some constant $C_0>0$ and all $j$.
Thus $|H\tilde\u_j|\leq |\tilde\u_j|^q\leq C_0^q|X|^q$ in $Q_1^-$.
Now if we apply Lemma \ref{holder-regularity} for each component of $\tilde\u_j=(\tilde u_j^1,\cdots,\tilde u_j^m)$, we obtain a caloric polynomial $P_j^i$ of degree at most $\lfloor 2+q\rfloor=2$ so that $|\tilde u_j^i(X)-P_j^i(X)|\leq C_1C_0|X|^{2+q}$ in $Q_1^-$ and the constant $C_1$ depends only on $n, q$ and an upper bound on $\lVert\tilde\u_j\rVert_{L^\infty(Q_1^-)}$. 
Since $(0,0)\in\Gamma^\kp(\u)$, so $P_j^i\equiv0$ and then $|\tilde u_j^i(X)|\leq C_1C_0|X|^{2+q}$.
By a bootstrap argument we find out the uniform estimate  $|\tilde\u_j(X)|\leq C_\epsilon|X|^{\kp-\epsilon}$ for every $\epsilon>0$. Therefore, we get
\begin{equation}\label{u0-vanishing derivatives}
\u_0(0,0)=\partial_t^i\partial_x^\mu\u_0(0,0)=0,\quad \text{ for all }2i+|\mu|<\kp.
\end{equation}
Obviously \eqref{u0-property} and \eqref{u0-vanishing derivatives}, along with the fact that $\kp\notin\mathbb{N}$, violates Liouville's theorem and we have a contradiction.

{\bf Case $\kp\in\mathbb{N}$:} 
Consider the function $\v=\eta\u$ where $ \eta \in C_0^\infty(B_{3/4})$ satisfies $0\leq\eta \leq 1$, 
and $\eta =1$ in $ B_{1/2}$. 
Fix $0< r<\frac12$,  let $ \rho_i:= 2^{-i}r, i =0,1,2,...$, and define $\v_{\rho_i}(x,t)= \v(\rho_ix,\rho_i^2t)/\rho_i^{\kp}$, then 
\begin{align}\label{lem3.2-eq1}
\int_{-r^2}^{0} \int_{B_r} |\u|^{1+q}G dxdt = &\sum_{i=1}^{\infty} \int_{-\rho_{i-1}^2}^{-\rho_i^2} \int_{B_{r} }|\v|^{1+q}G dxdt\notag\\
=&\sum_{i=1}^{\infty}\rho_i^{2\kp} \int_{-4}^{-1} \int_{B_{2^i} }|\v_{\rho_i}|^{1+q}G dxdt
\leq\sum_{i=1}^{\infty}\rho_i^{2\kp} \int_{-4}^{-1} \int_{\R^n }|\v_{\rho_i}|^{1+q}G dxdt\notag\\
=&\frac{1+q}2\sum_{i=1}^{\infty}\rho_i^{2\kp}\left(\W(\v_{\rho_i},1)-\int_{-4}^{-1} \int_{\R^n }\left(|\nabla\v_{\rho_i}|^2+\frac{\kp|\v_{\rho_i}|^2}{2t}\right)Gdxdt\right)\notag\\
=& \frac{1+q}2\sum_{i=1}^{\infty}\rho_i^{2\kp}\left(\W(\v,\rho_i)-\int_{-4}^{-1} \int_{\R^n }\left(|\nabla(\v_{\rho_i}-\p)|^2+\frac{\kp|\v_{\rho_i}-\p|^2}{2t}\right)Gdxdt\right)  \notag\\
\leq&\frac{1+q}2\sum_{i=1}^{\infty}\rho_i^{2\kp}\left(\W(\v,1)+F(1)+\int_{-4}^{-1} \int_{\R^n}\frac{\kp|\v_{\rho_i}-\p|^2}{-2t}Gdxdt\right)\notag\\
\leq&Cr^{2\kp}\left(1+\int_{-1}^{0} \int_{\R^n}\frac{\kp|\v_{r}-\p|^2}{-2t}Gdxdt\right),
\end{align} 
where we have used Lemma \ref{appex-lemma3}, $F$ is the function defined in Theorem \ref{cutoff} and  $\p\in\cH$, the space of all $\kp$-backward self-similar caloric vector-functions.
We now let  $\p = \pi_r$, where 
$$\pi_r={\rm argmin}_{\q\in\cH}\int_{-1}^{0} \int_{\R^n}\frac{|\v_{r}-\q|^2}{-t}Gdxdt ,$$
and  observe that 
\begin{equation}\label{eq:product}
\int_{-1}^{0} \int_{\R^n}\frac{(\v_{r}-\pi_r)\cdot\p}{-t}Gdxdt=0, \qquad\text{for every }\p\in\cH.
\end{equation}
Now suppose, towards a contradiction, that there is a sequence $r_k\rightarrow0$, such that 
$$
\sup_{Q_r^-}|\u|\leq kr^\kp,\qquad \forall r\geq r_k,\qquad \sup_{Q_{r_k}^-}|\u|=kr_k^\kp.
$$
Consider the scaling $\u_r(x,t)=\u(rx,r^2t)/r^\kp$, where  the sequence $\u_{r_k}$ satisfies the doubling condition \eqref{doubling} because
$$
\lVert\u_{r_k}\rVert_{L^\infty(Q_{2}^-)}=2^\kp\lVert\u_{2r_k}\rVert_{L^\infty(Q_{1}^-)}\leq 2^\kp k=2^\kp\lVert\u_{r_k}\rVert_{L^\infty(Q_{1}^-)}.
$$
Therefore  Lemma \ref{estimate-solution} and \eqref{lem3.2-eq1} implies that
$$M_k=\left(\int_{-1}^{0} \int_{\R^n}\frac{|\v_{r_k}-\pi_{r_k}|^2}{-t}Gdxdt\right)^{1/2}\longrightarrow\infty.$$
For $\w_k=\frac{\v_{r_k}-\pi_{r_k}}{M_k}$, we have 
\begin{equation}\label{lemma3:eq1}
\int_{-1}^{0} \int_{\R^n}\frac{|\w_k|^2}{-t}Gdxdt=1.
\end{equation}
Furthermore, we can show that $\{\nabla\w_kG^{1/2}\}$ is bounded in $L^2(-1,0;L^2(\R^n))$.
In order to show this, we can write 
\begin{align}\label{lemma3:eq2}
\int_{-1}^{0} \int_{\R^n }\left(|\nabla\w_{k}|^2+\frac{\kp|\w_{k}|^2}{2t}\right)Gdxdt
&=\frac1{M_k^2}\int_{-1}^{0} \int_{\R^n }\left(|\nabla\v_{r_k}|^2+\frac{\kp|\v_{r_k}|^2}{2t}\right)Gdxdt\notag\\
&=\frac1{M_k^2}\sum_{i=1}^{\infty}2^{-2i\kp} \int_{-4}^{-1} \int_{\R^n }\left(|\nabla\v_{2^{-i}r_k}|^2+\frac{\kp|\v_{2^{-i}r_k}|^2}{2t}\right)Gdxdt\notag\\
&\leq \frac1{M_k^2}\sum_{i=1}^{\infty}2^{-2i\kp} \W(\v,2^{-i}r_k)\longrightarrow0,
\end{align}
which together with  \eqref{lemma3:eq1} implies 
\begin{equation}\label{lemma3:eq3}
\int_{-1}^{0} \int_{\R^n }|\nabla\w_{k}|^2Gdxdt=O(1).
\end{equation}
On the other hand, we have 
\begin{equation}\label{eq:Hw}
H\w_k=\frac1{M_k}H\v_{r_k}=\frac1{M_k}\left(r_k^2\u_{r_k}\Delta\eta(r_kx)+\eta(r_kx)^{1-q}f(\v_{r_k})+2r_k\nabla\u_{r_k}\cdot\nabla\eta(r_kx)\right),
\end{equation}
and also,
\begin{align*}
\frac2{1+q} \int_{-1}^{0}\int_{\R^n}|\v_{r_k}|^{1+q}Gdxdt
&=\frac2{1+q}\sum_{i=1}^{\infty}2^{-2i\kp} \int_{-4}^{-1} \int_{\R^n }|\v_{2^{-i}r_k}|^{1+q}Gdxdt\notag\\
&=\sum_{i=1}^{\infty}2^{-2i\kp}\left(\W(\v_{2^{-i}r_k},1)-\int_{-4}^{-1} \int_{\R^n }\left(|\nabla\v_{2^{-i}r_k}|^2+\frac{\kp|\v_{2^{-i}r_k}|^2}{2t}\right)Gdxdt\right)\notag\\
&\leq \sum_{i=1}^{\infty}2^{-2i\kp}\left(\W(\v,1)+F(1)+\int_{-4}^{-1} \int_{\R^n }\frac{\kp|\v_{2^{-i}r_k}-\pi_{r_k}|^2}{-2t}Gdxdt\right)\notag\\
&=O(1)+\int_{-1}^{0}\int_{\R^n}\frac{\kp|\v_{r_k}-\pi_{r_k}|^2}{-2t}Gdxdt\notag\\
&=O(1)+\frac{M_k^2}{1-q}.
\end{align*}
Therefore by \eqref{eq:Hw}, we get for $q>0$
\begin{align*}
\int_{-1}^{0} \int_{\R^n }|H\w_k|^{(1+q)/q}Gdxdt\leq& \frac1{M_k^{(1+q)/q}}\left(C(\eta,\lVert\u\rVert_{H^1(B_1,\R^m)})+\int_{-1}^{0} \int_{\R^n } |\v_{r_k}|^{1+q}Gdxdt\right)\\
\leq&\frac1{M_k^{(1+q)/q}}\left(C(\eta,\lVert\u\rVert_{H^1(B_1,\R^m)},n,q)+\frac{M_k^2}{1-q}\right)\longrightarrow0,
\end{align*}
and $\lVert H\w_k\rVert_\infty\rightarrow0$ for $q=0$.
Hence $\{\w_k\}$ is bounded in $W^{1,p}(-1,\frac1R;W^{2,p}(B_R))$ for all fixed $R>0$ and by diagonalization technique there is a weakly convergent subsequence with limit $\w_0$ in $L^2(-1,0;W^{2,p}_{\rm loc}(\R^n))$, satisfying $H\w_0=0$. 
We claim next that  the strong convergence 
\begin{equation}\label{lemma3:eq3.5}
\frac{\w_k}{(-t)^{1/2}}G^{1/2}\rightarrow \frac{\w_0}{(-t)^{1/2}}G^{1/2},\qquad\text{ in }L^2(-1,0;L^2(\R^n)),
\end{equation}
holds, which follows  if we prove the  uniform convergence in $k$ 
\begin{equation}\label{lemma3:eq4}
\int_{-1}^{0} \int_{\R^n\setminus B_R }\frac{|\w_k|^2}{-t}Gdxdt\rightarrow 0,\qquad\text{ as }R\rightarrow\infty .
\end{equation}
This can obviously be obtained by applying Lemma \ref{appx-lemma2} and relations \eqref{lemma3:eq1} and \eqref{lemma3:eq3}
\begin{align*}
R^2\int_{-1}^{-0} \int_{\R^n\setminus B_R }\frac{|\w_k|^2}{-t}Gdxdt\leq \int_{-1}^{0} \int_{\R^n }|\w_k|^2\frac{|x|^2}{-t}Gdxdt=O(1).
\end{align*}
Therefore we get
$$
\int_{-1}^{0} \int_{\R^n}\frac{|\w_0|^2}{-t}Gdxdt=1,
$$
 and also by \eqref{eq:product},
\begin{equation}\label{eq:product-w0}
\int_{-1}^{0} \int_{\R^n}\frac{\w_0\cdot\p}{-t}Gdxdt=0, \qquad\text{for every }\p\in\cH.
\end{equation}
Moreover, from  \eqref{lemma3:eq2}, \eqref{lemma3:eq3.5} and weakly lower semicontinuity of norm we obtain for every $R>0$
\begin{align*}
\int_{-1}^{0} \int_{B_R }\left(|\nabla\w_0|^2+\frac{\kp|\w_0|^2}{2t}\right)Gdxdt
\leq\lim_{k\rightarrow\infty}\int_{-1}^{0} \int_{\R^n }\left(|\nabla\w_k|^2+\frac{\kp|\w_k|^2}{2t}\right)Gdxdt\leq0.
\end{align*}
It implies that
\begin{equation}\label{eq:arg}
0\geq\int_{-1}^{0} \int_{\R^n }\left(|\nabla\w_0|^2+\frac{\kp|\w_0|^2}{2t}\right)Gdxdt=
\sum_{i=1}^{\infty}2^{-2i\kp}\int_{-4}^{-1} \int_{\R^n }\left(|\nabla\w_{0,2^{-i}}|^2+\frac{\kp|\w_{0,2^{-i}}|^2}{2t}\right)Gdxdt.
\end{equation}
If we further have
\begin{equation}\label{eq:zero-diff}
|\partial_t^\ell\partial_x^\mu w_0^j(0,0)|=0,\qquad\text{for }2\ell+|\mu|\leq\kp-1,
\end{equation}
 then by  Lemma \ref{Nr}, each component $w_0^j$ of $\w_{0,2^{-i}}$ must satisfy 
$$
\int_{-4}^{-1} \int_{\R^n }\left(|\nabla w_0^j|^2+\frac{\kp|w_0^j|^2}{2t}\right)Gdxdt\geq0.
$$
Summing over $j$ and comparing with \eqref{eq:arg}, implies that $w_0^j$ is a $\kp$-backward self-similar caloric function. 
But \eqref{eq:product-w0} implies that $\w_0=0$ which contradicts \eqref{lemma3:eq4}.

To close the argument, we need to prove \eqref{eq:zero-diff}. This can be shown by invoking Lemma \ref{holder-regularity} to obtain uniform estimate  $\lVert\w_k\rVert_{L^\infty(Q_r^-)}=o(r^{\kp-1})$. 
To apply Lemma \ref{holder-regularity}
it is necessary to show  the uniform estimate $\lVert H\w_k\rVert_{L^\infty(Q_r^-)}=o(r^{\kp-2-1/2})$. 
Since we have assumed  $\lVert\u_{r_k}\rVert_{L^\infty(Q_{1}^-)}\rightarrow\infty$ by contradiction, the scaled sequence $\tilde\u_k:=\u_{r_k}/\lVert\u_{r_k}\rVert_{L^\infty(Q_{1}^-)}$ satisfies $H\tilde\u_k=f(\tilde\u_k)/\lVert\u_{r_k}\rVert_{L^\infty(Q_{1}^-)}^{1-q}\rightarrow0$ and converges to a caloric function $\tilde\u_0$ as a subsequence. 
Moreover, $|\tilde\u_k(X)|\leq C_0|X|$ for a constant $C_0>0$ and all $k$. Now apply Lemma \ref{holder-regularity} repeatedly to obtain the uniform estimate $|\tilde\u_k(X)|\leq C_\epsilon|X|^{\kp-\epsilon}$ for a small value $\epsilon$.
So, $|\u_{r_k}(X)|\leq C_\epsilon\lVert\u_{r_k}\rVert_{L^\infty(Q_1^-)}|X|^{\kp-\epsilon}$ and by Lemma \ref{estimate-solution} as well as \eqref{lem3.2-eq1}
$$
|H\w_k(X)|\leq \frac1{M_k}|f(\u_{r_k}(X))|\leq \frac{C_\epsilon^q}{M_k}(1+M_k^2)^{q/(1+q)}|X|^{\kp-2-\epsilon q}\leq C|X|^{\kp-2-\epsilon q}.
$$
\end{proof}

\begin{remark}\label{forward-regularity}
Although Theorem \ref{thm:growth} shows the backward regularity, we can see obviously the regularity in forward problem. 
A line of proof can be considered toward a contradiction and assuming the sequence $r_j\rightarrow0$ such that 
$$
\sup_{Q_r^+}|\emph{\u}|\leq jr^\kp,\qquad \forall r\geq r_j,\qquad \sup_{Q_{r_j}^+}|\emph{\u}|=jr_j^\kp.
$$
 Then $ \emph{\u}_j(X)=\emph{\u}(r_jX)/(jr_j^\kp)$ converges to a caloric function $\emph{\u}_0$ in $\R^n\times\R$ with polynomial growth, $\lVert\emph{\u}_0\rVert_{L^\infty(Q_1^+)}=1$ and $\emph{\u}_0\equiv0$ for $t\leq0$ (we also apply here Theorem \ref{thm:growth}). This contradicts the uniqueness of heat equation solution with polynomial growth in forward problem.
 \end{remark}


\section{Homogeneous global solutions}\label{sec:global-solution}

In this section we perform energy classification of regular free boundary points, that will be needed later  in order  to establish the H\"{o}lder regularity of the time derivative $ \partial_t \u $ in the next section.  Indeed the main goal is to show that half-space solutions are isolated within certain topologies. 
The proofs for the case $q=0$ and $q > 0$ differs to some extent and hence we are forced to consider them separately. For the  case $q=0$ we need to consider  two lemmas (Lemmas \ref{halfspace}, and \ref{closenesstohalfspace}) that will give us the result. The proof for the case  $q>0$ takes a different turn, and is shown in the proof of Proposition \ref{H-isolated}.

\begin{lemma} \label{halfspace}
Let $q=0$ and  $ \emph{\u}$ be a backward self-similar solution to \eqref{system}. If $ \{ |\emph{\u}|>0 \}\cap Q_1^-\subset \{ x_n>-\delta\}$, where $\delta >0$ is small, then $ \emph{\u}\in \mathbb{H}$.
\end{lemma}
\marginpar{\small \color{red}}

\begin{proof}
Let $\u$ be a backward self-similar solution, and recall  that  
the condition of homogeneity  (for each component) is 
\begin{equation}
L u^k := 2t \partial_t u^k +x \cdot \nabla u^k -2 u^k =0.
\end{equation}
Hence we obtain the following equation for each component
\begin{equation}
\frac{2 u^k-x \cdot \nabla u^k}{-2t}+\Delta u^k- \frac{u^k}{| \u|} \chi_{\{| \u|>0 \}}=0,
\end{equation}
and
\begin{equation}
2t \Delta u^k+x \cdot \nabla u^k-2t \frac{u^k}{| \u|} \chi_{\{| \u|>0 \}}=2 u^k.
\end{equation}

Denote by $ \L_0:=- \Delta + x \cdot \nabla $ and $\L:=\L_0+\frac{1}{|\u|}$. Then for 
$t=-\frac{1}{2}$, any $ u^k $ is an eigenfunction of $\L$ in $\{U>0\}$ corresponding to the eigenvalue $\lambda=2$.

\par
We want to show that  $\lambda =2  $ is the first eigenvalue for $\L$, since then $ u^k=c_k | \u|$ in each connected component of $\{|\u|>0\}$, and we have a scalar problem for $ |\u|$ when $ t =-1/2$.
It is sufficient to show that $2$ is not larger than  the second eigenvalue for $ \L_0$.

We prove that for some $ \delta >0$, $ \lambda_2(\L, \{x_n>-\delta \})>2$, which implies  $ \lambda_1(\L, \{x_n>-\delta \})=2. $
Since\footnote{This follows from a simple computation for one dimensional case, and the fact that eigenvalues decrease by symmetrisation, and translation invariance of the set $\R^n_+$ in directions orthogonal to $\e_n$.}
  $ \lambda_2(\L_0,\R^n_+)=3$,  we have 
\begin{equation}
 \lambda_2(\L, \{x_n>-\delta \})> \lambda_2(\L_0, \{x_n>-\delta \})\geq  \lambda_2(\L_0, \R^n_+)-\omega(\delta)=3-\omega(\delta),
\end{equation}
where $\omega(\delta)$ is the modulus of continuity of $  \lambda_2(\L_0, \{x_n>-\delta \})$. By choosing $ \delta >0$ small, we will obtain $ \lambda_2(\L, \{x_n>-\delta \})>2$,
implying that  $  \lambda_1(\L, \{x_n>-\delta \})=2$. Hence $ \u=\c |\u|$ in each connected component of $\{ |\u|>0\}$, where $ \c \in \R^m$
depends on the component and $ |\c|=1$ . It remains to observe that for $t= -1/2$, the function $U= |\u|$
is a homogeneous stationary solution to the equation $ H U =\chi_{\{ U >0\}}$, and therefore $U  $ is a half-space solution and $\u \in \mathbb{H}$. 
\end{proof}

\begin{lemma} \label{closenesstohalfspace}
(Closeness to half-space) 
Let $q=0$ and    $\emph{\u}$ be a backward self-similar solution to the system \eqref{system} with the property 
\begin{equation}
\iint_{Q_1^-} |\emph{\u}-\emph{\h}|Gdx dt <\varepsilon,
\end{equation}
where $ \emph{\h}= \frac{(x_n^+)^2}{2} \emph{\e}_1$.
Then 
\begin{equation}\label{epsilonbeta}
\{ |\emph{\u}|>0\} \cap Q_{1/2}^- \subset \{ (x,t): x_n >-C\varepsilon^{\beta}\}.
\end{equation}
for  $ C=C(n,m)$, and $\beta = \beta (n)$.
\end{lemma}

\begin{proof}
The proof is standard and follows from the nondegeneracy.
Let $(x_0,t_0) \in  \{ |\u|>0\} \cap  Q_{1/2}^-$, and $x_n^0=-\varrho <0$, then
\begin{equation}
\iint_{Q_\varrho^-} |\u|Gdx dt \leq \iint_{Q_1^-} |\u-\h| G dx dt \leq \varepsilon.
\end{equation}
By the nondegeneracy, there exists $X \in Q_\varrho^-(x_0,t_0)$, such that 
$$|\u(X)|=\sup_{Q_\varrho^-(x_0,t_0)} |\u|\geq c_n \varrho^2.$$
Then for a small $r>0$, 
$$ \inf_{Q_r^-(X)} |\u|\geq c_n \varrho^2-C_n r^2 \geq C \varrho ^2, $$
and
\begin{equation}
\varepsilon \geq \iint_{Q_r^-(X)}  |\u| Gdx dt \geq C \varrho^{n+4}.
\end{equation}
Now \eqref{epsilonbeta} follows with $\beta =\frac{1}{n+4}$.
\end{proof}
The next proposition shows that the half-space solutions in $\bH$ are isolated in the class of $\kp$-backward self-similar solutions. 
 
\begin{proposition}\label{H-isolated}
The half-space solutions are isolated (in the topology of $L^{2}(-1,0;H^1(B_1;\R^m))$) within the class of backward self-similar solutions of degree $\kp$.
\end{proposition}
\begin{proof}
The proof follows from Lemmas  \ref{halfspace} and \ref{closenesstohalfspace} for $q=0$. When $q>0$, we assume toward a contradiction that there exists a sequence of backward self-similar solutions of degree $\kp$, say $\u_i$, such that 
$$0<\inf_{\h\in\mathbb{H}}\lVert\u_i-\h\rVert_{L^{2}(-1,0;H^1(B_1;\R^m))}=\lVert\u_i-\hat\h\rVert_{L^{2}(-1,0;H^1(B_1;\R^m))}=:\delta_i\rightarrow0,\quad\text{as }i\rightarrow\infty,$$
where $\hat\h=\alpha(x_n^+)^\kp\e_1$. When passing to a subsequence, $(\u_i-\hat\h)/\delta_i=:\w_i\rightharpoonup\w$
weakly in $L^{2}(-1,0;H^1(B_1;\R^m))$, the limit $\w$ is still a backward self-similar function of degree $\kp$.

Furthermore, for $\phi\in C_0^\infty(Q_1^-;\R^m)$ we have
\begin{align*}
\int_{-1}^0\int_{B_1}-\nabla\w_i:\nabla\phi+\w_i\cdot\partial_t\phi\,dxdt=&\frac1{\delta_i}\int_{-1}^0\int_{B_1}\left(f(\u_i)-f(\hat\h)\right)\cdot\phi\,dxdt\\
=&\frac1{\delta_i}\int_{-1}^0\int_{B_1}\int_0^1\frac d{d\tau}f(\hat\h+\tau(\u_i-\hat\h))\cdot\phi\,d\tau dxdt\\
=&\int_{-1}^0\int_{B_1}\int_0^1f_{\u}(\hat\h+\tau\delta_i\w_i)(\w_i)\cdot\phi\,d\tau dxdt
\end{align*}
If ${\rm supp}\,\phi\subset B_1^-\times(-1,0)$, we conclude that
\begin{align*}
\int_{-1}^0\int_{B_1^-}-\nabla\w_i:\nabla\phi+\w_i\cdot\partial_t\phi\,dxdt=&\int_{-1}^0\int_{B_1^-}\int_0^1f_{\u}(\hat\h+\tau\delta_i\w_i)(\w_i)\cdot\phi\,d\tau dxdt\\
=&\int_{-1}^0\int_{B_1^-}\int_0^1f_{\u}(\tau\delta_i\w_i)(\w_i)\cdot\phi\,d\tau dxdt\\
=&\frac1q\delta_i^{q-1}\int_{-1}^0\int_{B_1^-}f_{\u}(\w_i)(\w_i)\cdot\phi\,dxdt,
\end{align*}
let $i\rightarrow\infty$ to obtain
$$\int_{-1}^0\int_{B_1^-}f_{\u}(\w)(\w)\cdot\phi\,dxdt=0.$$
Then $\w\equiv0$ in $B_1^-\times(-1,0)$. Now for every ${\rm supp}\,\phi\subset B_1^+\times(-1,0)$,
$$
\int_{-1}^0\int_{B_1^+}-\nabla\w:\nabla\phi+\w\cdot\partial_t\phi\,dxdt=\int_{-1}^0\int_{B_1^+}f_{\u}(\hat\h)(\w)\cdot\phi\,dxdt.
$$
Thus $H\w=f_\u(\hat\h)(\w)$ in $B_1^+\times(-1,0)$. Now let $w^j:=\w\cdot\e_j$ for $1\leq j\leq m$, then
$$Hw^j=q\kp(\kp-1)(x_n^+)^{-2}w^j, \quad\text{ for }j=1,$$
and
$$Hw^j=\kp(\kp-1)(x_n^+)^{-2}w^j, \quad\text{ for }j>1.$$
Next extend $w^j$ to a backward self-similar function of degree $\kp$ in $\{x_n<0\}$ and define
$$\tilde w^j(x',x_n,t):=\left\{\begin{array}{ll}w^j(x',x_n,t),&x_n>0,\\[10pt]
-w^j(x',x_n,t),&x_n<0,\end{array}\right.$$
which is a backward self-similar weak solution of degree $\kp$ and satisfies 
\begin{equation}\label{eq-w_j}
H\tilde w^j=\left\{\begin{array}{ll}q\kp(\kp-1)|x_n|^{-2}\tilde w^j,&\text{for }j=1,\\[10pt]
\kp(\kp-1)|x_n|^{-2}\tilde w_j,&\text{for }j>1.\end{array}\right.
\end{equation}
If we consider any multiindex $\mu\in \mathbb{Z}_+^{n-1}\times\{0\}$ and any nonnegative integer $\gamma\in \mathbb{Z}_+$ as well as the higher order partial derivatives $\partial_t^\gamma\partial_x^\mu\tilde w^j=:\zeta$ then $\zeta$ is a backward self-similar function of order $\kp-|\mu|_1-2\gamma$ and satisfies again in the same equation in $\R^n\times(-\infty,0)$.
From the integrability and homogeneity we infer that $\partial_t^\gamma\partial_x^\mu\tilde w^j\equiv0$ for $\kp-|\mu|_1-2\gamma+1\leq -n/2$.
Thus $(x',t)\mapsto \tilde w^j(x',x_n,t)$ is a polynomial and the homogeneity  imply the existence of a polynomial $p$  such that $ w^j(x',x_n,t)=x_n^\kp w^j(\frac{x'}{x_n},1,\frac{t}{x_n^2})=x_n^\kp p(\frac{x'}{x_n},\frac{t}{x_n^2})$ for $x_n>0$. 
Next choose $\gamma$ such that $\partial_t^\gamma p=r(x')\neq0$, then  according to the $H^1$-integrability of $\partial_t^\gamma w^j=x_n^{\kp-2\gamma}r(\frac{x'}{x_n})$ we know that  $\kp-2\gamma-{\rm deg}\,r>\frac12$.
Take the multiindex $\mu\in\mathbb{Z}_+^{n-1}\times\{0\}$ such that $|\mu|_1={\rm deg}\,r$ and $\partial_x^\mu r\neq0$, and insert $\partial_t^\gamma\partial_x^\mu w^j=\partial_x^\mu r x_n^{\kp-2\gamma-|\mu|_1}$ in equation \eqref{eq-w_j}, which implies that
$$\begin{array}{ll}
(\kp-2\gamma-|\mu|_1)(\kp-2\gamma-|\mu|_1-1)=q\kp(\kp-1),&\text{for }j=1,\\[10pt]
(\kp-2\gamma-|\mu|_1)(\kp-2\gamma-|\mu|_1-1)=\kp(\kp-1),&\text{for }j>1,
\end{array}
$$
and hence 
$$\begin{array}{ll}
2\gamma+|\mu|_1=1,\text{ or }2\kp-2,&\text{for }j=1,\\[10pt]
2\gamma+|\mu|_1=0,\text{ or }2\kp-1,&\text{for }j>1.
\end{array}
$$
The condition $\kp-2\gamma-{\rm deg}\,r>\frac12$ and $\kp>2$ yields that the only possible case is $\gamma=0$ and $|\mu|_1=1$ for $j=1$ and $|\mu|_1=0$ for $j>1$.
We obtain that $w^1(x,t)=x_n^\kp(d+\ell\cdot x'/x_n)$ and $w^j(x,t)=\ell_jx_n^\kp$ for $j>1$. Comparing with the equation \eqref{eq-w_j} implies that we must have $d=0$. 
To sum up, we find that $\w(x,t)=(x_n^{\kp-1}\ell_1\cdot x',\ell_2 x_n^\kp,\cdots,\ell_m x_n^\kp)$ for some $\ell_1\in\R^{n-1}$ and $\ell_2,\cdots,\ell_m\in\R$.

Recall that we have chosen $\hat \h$ as the best approximation of $\u_i$ in $\mathbb{H}$.  
So, it follows that for $\h_\nu(x):=\alpha\max(x\cdot\nu,0)^\kp\e_1$,
\begin{equation}\label{best-approx}
(\w_i,\h_\nu-\h)_{L^{2}(-1,0;H^1(B_1;\R^m)))}\leq\frac1{2\delta_i}\lVert\h_\nu-\h\rVert^2_{L^{2}(-1,0;H^1(B_1;\R^m))}.
\end{equation}
Now let $\nu\rightarrow\e_n$, so that $\frac{\nu-\e_n}{|\nu-e_n|}$ converges to the vector $\xi$ (where $\xi\cdot\e_n=0$), then
\begin{align*}
o(1)\geq\int_{-1}^0\int_{B_1}&(\w_i\cdot\e_1)\kp(x_n^+)^{\kp-1}(x\cdot\xi)+\\
&\nabla(\w_i\cdot\e_1)\cdot\left[\kp(x_n^+)^{\kp-1}\xi+\kp(\kp-1)(x_n^+)^{\kp-2}(x\cdot\xi)\e_n\right]\,dxdt.
\end{align*}
Choosing $\xi=(\ell_1,0)$ and passing to the limit in $i$, we obtain that 
\begin{align*}
0\geq\int_{-1}^0\int_{B_1}&\kp(x_n^+)^{2\kp-2}(x'\cdot\ell_1)^2+\kp(x_n^+)^{2\kp-2}|\ell_1|^2\\
&+\kp(\kp-1)^2(x_n^+)^{2\kp-4}(x'\cdot\ell_1)^2\,dxdt.
\end{align*}
Hence, $\ell_1=0$, and then $\w\cdot\e_1=0$.

If we apply once more the relation \eqref{best-approx} for $\h_\theta=\alpha(x_n^+)^\kp\e_t$ istead of $\h_\nu$, where $\e_\theta=(\cos\theta)\e_1\pm(\sin\theta)\e_j$, and let $\theta\rightarrow0$. We obtain
$$(\w_i,\pm\alpha(x_n^+)\kp\e_j)_{W^{1,2}(Q_1^-;\R^m)}\leq0.$$
Therefore,
$$\ell_j\lVert(x_n^+)^\kp\rVert^2_{W^{1,2}(Q_1^-;\R^m)}=0,$$
and then $\ell_j=0$.

So far, we have proved that $\w\equiv0$. 
In order to obtain a contradiction to the assumption $\lVert\w_i\rVert_{L^{2}(-1,0;H^1(B_1;\R^m))}=1$, it is therefore sufficient to show the strong convergence of $\nabla\w_i$ to $\nabla\w$  in $L^{2}(-1,0;L^2(B_1;\R^m))$ as a subsequence $i\rightarrow\infty$. But by compact imbedding on the boundary
\begin{align*}
\int_{-1}^0\int_{B_1}|\nabla\w_i|^2\,dxdt=&\int_{-1}^0\int_{\partial B_1}\w_i\cdot(\nabla\w_i\cdot x)\,d\cH^{n-1}dt-\int_{-1}^0\int_{B_1}\w_i\cdot\Delta\w_i\,dxdt\\
=&\int_{-1}^0\int_{\partial B_1}\kp|\w_i|^2-2t\partial_t\w_i\cdot\w_i\,d\cH^{n-1}dt-\int_{-1}^0\int_{B_1}\w_i\cdot\partial_t\w_i\,dxdt\\
&-\frac1{\delta_i^2}\int_{-1}^0\int_{B_1}(\u_i-\hat\h)\cdot(f(\u_i)-f(\hat\h))\,dxdt\\
\leq&\int_{-1}^0\int_{\partial B_1}\kp|\w_i|^2-t\partial_t|\w_i|^2\,d\cH^{n-1}dt-\frac12\int_{-1}^0\int_{B_1}\partial_t|\w_i|^2\,dxdt\\
=&\int_{-1}^0\int_{\partial B_1}(\kp+1)|\w_i|^2\,d\cH^{n-1}dt-\int_{\partial B_1}|\w_i(x,-1)|^2\,d\cH^{n-1}\\
&+\frac12\int_{B_1}|\w_i(x,-1)|^2-|\w_i(x,0)|^2\,dx\\
\leq&\int_{-1}^0\int_{\partial B_1}(\kp+1)|\w_i|^2\,d\cH^{n-1}dt+\frac12\int_{B_1}|\w_i(x,-1)|^2\,dx\\
=&\int_{-1}^0\int_{\partial B_1}(\kp+1)|\w_i|^2\,d\cH^{n-1}dt+\frac1{n+2\kp+2}\int_{-1}^0\int_{B_1}|\w_i|^2dxdt\rightarrow0,
\end{align*}
as a subsequence $i\rightarrow\infty$. (Note that we have used the homogeneity property of $\w_i$ in the last line.)
\end{proof}

\begin{definition}(Regular points)
We say that a point $ z=(x,t) \in \Gamma^\kp(\emph{\u})$ is a regular\footnote{
They are also called low-energy points. When $q=0$, these points have the lowest energy (see Theorem \ref{lowenergy}). It is not generally true for $q>0$ and we have the lowest energy when the coincidence set has nonempty interior (see Theorem  \ref{lower-energy-for-q}). In this case, half-space solutions have the energy $\M(\h)=\frac{\alpha^{1+q}}{\kp(\kp-1)}\frac{4^\kp-1}{\sqrt\pi}2^{2\kp-3}\Gamma(\kp-\frac12)$, where $\Gamma$ is the Gamma function here. 
The time-dependent global solution $\theta(x,t)=(\frac{-2t}\kp)^{\kp/2}\e$ has the energy $\M(\theta)=\frac{2^{\kp-1}}{\kp^\kp(\kp-1)}(4^\kp-1)$ which is the less than the energy of half-space solutions for $\kp\geq5/2$.
} 
free boundary point for $\emph{\u}$ if 
at least one blowup limit of $\emph{\u}$ at $z$ belongs to $\bH$. We denote by $\mathcal{R}$ the set of all regular free boundary points in  $\Gamma(\emph{\u})$. 
\end{definition}

\begin{proposition}\label{Regular-well-define}
If $z_0=(x_0,t_0)\in\mathcal{R}$, then all blowup limits of $\emph{\u}$ at $z_0$ belong to $\bH$.
\end{proposition}

\begin{proof}
Suppose there are two sequences $r_i,\rho_i\rightarrow0$ such that the scaling $\u(x_0+r_i\cdot,t_0+r_i^2\cdot)/r_i^\kp$ and $\u(x_0+\rho_i\cdot,t_0+\rho_i^2\cdot)/\rho_i^\kp$ converges respectively to $\u_0\in\bH$ and a $\tilde \u_0\notin\bH$. Furthermore, we can assume that $r_{i+1}<\rho_i<r_i$.
By a continuity argument we can find $\rho_i<\tau_i<r_i$ when $i$ is large enough such that $\textsl{dist}\left(\u(x_0+\tau_i\cdot,t_0+\tau_i^2\cdot)/\tau_i^\kp,\bH\right)=\theta\textsl{dist}(\tilde\u_0,\bH)$ for an arbitrary  $\theta\in(0,1/2)$.
The boundedness of $\u(x_0+\tau_i\cdot,t_0+\tau_i^2\cdot)/\tau_i^\kp$ implies that every limit $\u^*$ of that is  a $\kp$-backward self-similar solution such that $\textsl{dist}(\u^*,\bH)=\theta\textsl{dist}(\tilde\u_0,\bH)$, which  for small $\theta$ contradicts the isolation property Proposition \ref{H-isolated}.
\end{proof}

\begin{remark}
We will conclude later the uniqueness of blowups at regular points of free boundary by Theorem \ref{converge-speed} in Section \ref{sec:regularity}.
\end{remark}

The following  lemma and theorem shows that the half-space solutions have the lowest energy among the global self-similar solutions for the case $q=0$.

\begin{lemma}\label{highenergy}
Let $q=0$ and  $\emph{\u}$ be a backward self-similar solution to the system \eqref{system}, satisfying 
 $|\emph{\u}|>0$ a.e.. Then  
 \begin{equation*}
 \M(\emph{\u})= \int_{-4}^{-1} \int_{\mathbb{R}^n}  |\emph{\u} | Gdx dt\geq \frac{15}{2}.
 \end{equation*}
\end{lemma}

\begin{proof}
Observe that  by homogeneity of $\u$ we have  $\M(\u)=\W(\u, r)$, for any $r $. 
 Integration by parts, and (again) homogeneity of $\u$ implies 
\begin{equation*}
\M(\u)=\int_{-4}^{-1} \int_{\mathbb{R}^n}  \left(   | \nabla \u |^2 + \frac{|\u|^2}{t } + 2 |\u | \right) Gdx dt=
\int_{-4}^{-1} \int_{\mathbb{R}^n}  |\u | Gdx dt.
\end{equation*}
Let $U=|\u|>0$ a.e., as we observed  before,
$ \Delta U-\partial_t U\geq 1$.
Hence 
\begin{equation}\label{henergy}
\begin{aligned}
3 =&\int_{-4}^{-1} \int_{\mathbb{R}^n} Gdx dt \leq \int_{-4}^{-1} \int_{\mathbb{R}^n}  ( \Delta U-\partial_t U)Gdx dt \\
=& \int_{-4}^{-1} \int_{\mathbb{R}^n} -  U \Delta G+U\partial_t  G-\partial_t (UG)dx dt  =- \int_{\mathbb{R}^n}  U(x,t) G(x,t)dx |_{t=-4}^{t=-1} \\
 =&-\pi^{-\frac{n}{2}}\int_{\mathbb{R}^n}  U(\sqrt{-4t} z,t) e^{-|z|^2}dz |_{t=-4}^{t=-1} =+4\pi^{-\frac{n}{2}}\int_{\mathbb{R}^n} t U( z,-1/4) e^{-|z|^2}dz |_{t=-4}^{t=-1} \\
 =& 12\pi^{-\frac{n}{2}}\int_{\mathbb{R}^n} U( z,-1/4) e^{-|z|^2}dz. 
\end{aligned}
\end{equation}
Employing again the homogeneity of $U$, we obtain 
\begin{equation*}
\begin{aligned}
\M(\u)= \int_{-4}^{-1} \int_{\mathbb{R}^n} U(x,t) G(x,t)dx dt=&-4\pi^{-\frac{n}{2}} \int_{-4}^{-1} \int_{\mathbb{R}^n} t  U( z,-1/4) e^{-|z|^2} dz dt \\
=&30 \pi^{-\frac{n}{2}} \int_{\mathbb{R}^n}  U( z,-1/4) e^{-|z|^2} dz \geq \frac{15}{2},
\end{aligned}
\end{equation*}
where we used \eqref{henergy} in the last step.
\end{proof}

\begin{theorem}\label{lowenergy}
Let  $q=0$ and $\emph{\u}$ be a backward self-similar solution to the system \eqref{system}.
Then $\M(\emph{\u}) \geq \frac{15}{4}$ and the equality holds if and only if $ \emph{\u} \in \mathbb{H}$.
\end{theorem}

\begin{proof}
\textit{\textbf{Step 1:}} We show that $\M(\u) = \frac{15}{4}$ if $\u$ is a half-space solution.
If $\u \in \mathbb{H}$, then 
\begin{equation*}
\M(\u)= 
\int_{-4}^{-1} \int_{\mathbb{R}^n}  |\u | Gdx dt=\frac{1}{2}\int_{-4}^{-1} -t dt=\frac{15}{4}.
\end{equation*}

\bigskip

\textit{\textbf{Step 2:}} If $U=|\u|>0$ a.e., then  $\M(\u) \geq  \frac{15}{2}$ by Lemma \ref{highenergy}.
Suppose that  $|\{U=0\}|>0$ and $\M(\u)<\frac{15}{4}$. 
By the nondegeneracy and quadratic decay estimates, we imply that the interior of $\{\u\equiv0\}$ is nonempty. 
Then we may choose $Q_r^-(Y)\subset \{\u\equiv 0\}$ in such a way that there exists a point $Z\in \partial_pQ_r^-(Y)\cap\Gamma(\u)$.
Moreover, since $\u$ is backward self-similar we can assume that the point $Z$ is very close to the origin and satisfies $\mathbb{W}(\u, 0+; Z)<15/4$, by the upper semicontinuity of the balanced energy. 
Hence  any blow-up at $Z$ is a half-space solution by Lemma \ref{halfspace} which contradicts the result in Step 1.

\bigskip

\textit{\textbf{Step 3:}} It remains to show that  if $\M(\u) = \frac{15}{4}$ for a backward self-similar solution $\u$, then  $\u$ is a halfspace solution. 
Let $X_0=(x_0,t_0)\in \Gamma$ be as in \textit{Step 2}, i.e.  such that $ Q_{r_0}^-(Y)\subset \{\u\equiv 0\}$, for a small $r>0$.  If $ X_0=0$, then $ \u \in \bH$ by Lemma \ref{halfspace}.
Assume that $| X_0|>0$. Let $ \u_0$ be a blow-up of $\u$ at $X_0$, then $ \u_0 \in \bH$.  Hence
\begin{equation}\label{homwrtz}
\frac{15}{4}=\M(\u_0)=\W(\u, 0+; X_0)\leq \W(\u, +\infty;X_0)=\W(\u ,+\infty;0)=\M(\u)=\frac{15}{4},
\end{equation}
since $\W(\u, +\infty,X)$ does not depend on $ X\in \Gamma$.
Indeed, by homogeneity of $\u $ at the origin, 
$$ \frac{\u(rx+x_0, r^2t+t_0)}{r^2}=\u(x+\frac{x_0}{r},t+\frac{t_0}{r^2}) \rightarrow \u(x,t), \textrm{ as } r\rightarrow +\infty.$$

\par
Therefore \eqref{homwrtz} implies that $ \W(\u,r;X_0)$ does not depend on $ r $ and $ \u$ is backward self-similar with centre at $X_0$, hence
$\u(x_0+rx, t_0+r^2 t)=r^2 \u(x_0+x,t_0+t)$. On the other hand, $\u(x_0+rx,t_0+r^2t)=r^2\u(x_0/r+x,t_0/r^2+t)$, and therefore
$\u(x_0/r+x,t_0/r^2+t)= \u(x_0+x,t_0+t)$. Letting $r\rightarrow +\infty$, we obtain $\u(x,t)= \u(x_0+x,t_0+t)$, for any $(x,t)$, hence 
$\u$ satisfies the assumptions in Lemma \ref{halfspace}, and $\u\in \bH$.
\end{proof}

The next theorem is a version of Theorem \ref{lowenergy} for case $q>0$ to show that half-space solutions have the lowest energy among the backward self-similar solutions whose coincidence set has nonempty interior.
\begin{theorem}\label{lower-energy-for-q}
Let $\emph{\u}$ be a backward self-similar solution to the system \eqref{system} 
satisfying $ \{|\emph{\u}|=0\}^\circ\neq\emptyset$. Then $\M(\emph{\u}) \geq A_q$, and equality implies that $\emph{\u}$ is a half-space solution; here $A_q=\M(\emph{\h})$ for every $\emph{\h}\in\bH$. 
\end{theorem}
\begin{proof}
The proof is indirect. Consider the self-similar solution $\u$ with $\M(\u)<A_q$.
Assume that $\{|\u|=0\}$ contains the cube $Q$ and  $X_0=(x_0,t_0)\in \partial Q\cap\partial\{|\u|>0\}$ and $t_0<0$. 
From here we deduce that all derivatives of $\u$ at $X_0$ vanish if they exist. If we start with the initial regularity and the estimate $|\u(X)| \leq C|X-X_0|$, we are able to apply  Lemma \ref{holder-regularity} iteratively and obtain  $|\u(X)| \leq C_\epsilon|X-X_0|^{\kp-\epsilon}$. This  implies that $X_0\in\Gamma^\kp$.
Also from self-similarity of $\u$, we infer that
$$
\W(\u,0+;X_0)=\lim_{r\rightarrow0^+}\W(\u,r;X_0)=\lim_{r\rightarrow0^+}\W(\u,\frac rm;X_0^m)
=\W(\u,0+;X_0^m),
$$
where $X_0^m:=(\frac{x_0}m,\frac{t_0}{m^2})$.
By the upper semicontinuity of the function $X\mapsto\W(\u,0+;X)$, we get
$$
\W(\u,0+;X_0)=\limsup_{m\rightarrow\infty}\W(\u,0+;X_0^m)\leq \W(\u,0+;0)\leq \W(\u,1;0)=\M(\u)<A_q.
$$
Thus every blow-up limit $\u_0$ of $\u$ at the point $X_0$ satisfies the inequality $\M(\u_0)<A_q$. Note that by the nondegeneracy property $\u_0\not\equiv0$. Now the self-similarity of $\u$ tells us that $\u_0$ must be time-independent. 
To see that let $\u_r(x,t):=\u(x_0+rx,t_0+r^2t)/r^\kp$ which converges to $\u_0$ in some sequence. 
According to the self-similarity of $\u$, we have
$$
\nabla\u(x_0+rx,t_0+r^2t)\cdot (x_0+rx)+2(t_0+r^2t)\partial_t\u(x_0+rx,t_0+r^2t)=\kp\u(x_0+rx,t_0+r^2t),
$$
so,
$$
r\nabla\u_r(x,t)\cdot (x_0+rx)+2(t_0+r^2t)\partial_t\u_r(x,t)=r^2\kp\u_r(x,t),
$$
and passing to the limit, we obtain $t_0\partial_t\u_0(x,t)=0$. Therefore, $\u_0$ is a $\kp$-homogeneous global solution of $\Delta\u=f(\u)$ and violates Proposition 4.6 in \cite{FSW20}, the elliptic version of this theorem.  To find the elliptic energy of  $\u_0$, we can write for every  $\kp$-homogeneous time-independent solution $\v$,
\begin{align*}
\M(\v)=&\frac{1-q}{1+q}\int_{-4}^{-1}\int_{\R^n}|\v(x)|^{1+q}G(x,t)dxdt=\frac{1-q}{1+q}\pi^{-\frac n2}\int_{-4}^{-1}\int_{\R^n}|\v(\sqrt{-4t}z)|^{1+q}e^{-|z|^2}dzdt\\
=&\frac{1-q}{1+q}\frac{4^{2\kp}-4^\kp}{4\kp}\pi^{-\frac n2}\int_{\R^n}|\v(z)|^{1+q}e^{-|z|^2}dz\\
=&\frac{1-q}{1+q}\frac{4^{2\kp}-4^\kp}{4\kp}\pi^{-\frac n2}\int_0^\infty\int_{\partial B_1}|\v(\hat z)|^{1+q}r^{\kp(1+q)+n-1}e^{-r^2}d\hat zdr\\
=&\frac{1-q}{1+q}c_{q,n}\int_{\partial B_1}|\v(\hat z)|^{1+q}d\hat z\\
=&\frac{1-q}{1+q}\tilde c_{q,n}\int_{B_1}|\v(z)|^{1+q}dz= \tilde c_{q,n} M(\v),
\end{align*}
where $M(\v)$ is the adjusted energy for the elliptic case which is used in Proposition 4.6 in \cite{FSW20}. In particular, for  $\h\in\bH$, we find out $M(\u_0)<M(\h)$. 

Finally, to prove the second part of the statement we consider the backward self-similar solution $\u$ satisfying $\M(\u)=A_q$, $Q\subset\{|u|=0\}$ and $X_0=(x_0,t_0)\in\partial Q\cap\partial\{|\u|>0\}$ for some $t_0<0$.
As in the first part of the proof we obtain that every blow-up limit $\u_0$ of $\u$ at the point $X_0$ satisfies the inequality $\M(\u_0)\leq A_q$, that $\u_0$ is a $\kp$-homogeneous time-independent solution and $\{|\u_0|=0\}^\circ\neq\emptyset$.
Thus according to Proposition 4.6 in \cite{FSW20}, $\u_0$ must be a half-space solution. Therefore, every blow-up limit of $\u$ at the point $X_0^m=(\frac{x_0}m,\frac{t_0}{m^2})$ must be a  half-space solution. Assuming $\u\notin\bH$, we find by a continuity argument for an arbitrary $\theta\in(0,1)$ a sequence $\rho_m\rightarrow0$ such that
$$
\textsl{dist}(\rho_m^{-\kp}\u(x_0/m+\rho_m\cdot, t_0/m^2+\rho_m^2\cdot),\bH)=\theta\textsl{dist}(\u,\bH)>0.
$$
It follows that $\u(x_0/m+\rho_m\cdot, t_0/m^2+\rho_m^2\cdot)/\rho_m^\kp$ converges to a backward self-similar solution $\u^*$ along a subsequence as $m\rightarrow\infty$, because
$$
\W(\u^*,r;0)=\lim_{m\rightarrow\infty}\W(\u,r\rho_m;X_0^m)\geq \W(\u,0+,X_0^m)=A_q,
$$
and for every $0<\rho$
$$
\W(\u^*,r,0)=\lim_{m\rightarrow\infty}\W(\u,r\rho_m;X_0^m)\leq \lim_{m\rightarrow\infty}\W(\u,\rho\,;X_0^m)= \W(\u,\rho\,;0).
$$
Then $\W(\u^*,r;0)=A_q$ for all $r>0$ and $\u^*$ must be a self-similar solution.
The conclusion is that $\textsl{dist}(\u^*,\bH)=\theta\textsl{dist}(\u,\bH)$ which for small $\theta$ contradicts the isolation property Proposition  \ref{H-isolated}.
\end{proof}

Here, we show that the regular points are an open set in $\Gamma(\u)=\partial\{|\u|>0\}$.

\begin{proposition}\label{R-points}
The regular set $\mathcal{R}$  is open relative to $\Gamma(\emph{\u})$.
\end{proposition}
\begin{proof}
Assume that there is a sequence $X_i=(x_i,t_i)\in \Gamma(\u)\setminus\mathcal{R}$ converging to $X_0=(x_0,t_0)\in \mathcal{R}$. 
We can find a sequence $\tau_i\rightarrow0$ and a subsequence of $X_i$ such that\footnote{The distance ranges between almost zero to infinity, depending on $\tau_i$.}
\begin{equation}\label{openness-rel2}
\textsl{dist}\left(\u(x_i+\tau_i\cdot\,,t_i+\tau_i^2\cdot)/\tau_i^\kp,\bH\right)=\frac{c}{2^{2\kp+1}},
\end{equation}
 where $c$ is the constant defined in Proposition \ref{Nondegeneracy} and the distance is measured in $L^\infty(Q_1^-)$.
The uniform boundedness of set $\bH$ implies the convergence $\u(x_i+\tau_i\cdot\,,t_i+\tau_i^2\cdot)/\tau_i^\kp$ in a subsequence to a global solution $\u^*$. 
For convenience assume that $\lVert\u_{\tau_i,X_i}-\h\rVert_{L^\infty(Q_1^-)}\leq c/4^\kp$ for $\h(x,t)=\alpha(x^1_+)^\kp$.
Then 
\begin{equation}\label{openness-rel1}
|\u_{\tau_i,X_i}(x,t)|\leq \frac c{4^\kp}
\end{equation}
 for all $(x,t)\in Q_1^-$ where $x_1\leq0$.  
According to the nondegeneracy property, Proposition \ref{Nondegeneracy}, we know that
$\sup_{Q_{r}^-(Z)}|\u_{\tau_i,X_i}|\geq cr^\kp$  for all $Z\in\overline{\{|\u_{\tau_i,X_i}|>0\}}$ such that $Q_{r}^-(Z)\subseteq Q_1^-$.
Comparing with \eqref{openness-rel1} for $r>1/4$, we deduce that $\u_{\tau_i,X_i}\equiv0$ in $\{(x,t)\in Q_{1/2}^-: x_1\leq -1/4\}$.
Therefore, the coincidence set $\{|\u^*|=0\}$ has a nonempty interior and there exists cube $Q\subseteq\{|\u^*|=0\}$ and $Y_0\in \partial Q\cap\partial\{|\u^*|>0\}$. According to Theorem \ref{lowenergy} and Theorem \ref{lower-energy-for-q}, $Y_0$ is a regular point for $\u^*$ provided its energy is not larger than $A_q$. To see this, we  fix  $r\leq1$ and consider the energy value  
\begin{align*}
  \W(\u^*,0+;Y_0)\leq \W(\u^*,r;Y_0)&=\lim_{i\rightarrow\infty}\W(\u,r\tau_i;X_i+\tau_iY_0)  
\leq \lim_{i\rightarrow\infty}\W(\u,\rho;X_i+\tau_iY_0)
=\W(\u,\rho;X_0),
\end{align*}
where $\rho>0$ is again an arbitrary constant. Then $\W(\u^*,0+;Y_0) \leq \W(\u,0+;X_0) \leq A_q$ and $Y_0$ is a regular point. 
So, $\W(\u^*,0+;Y_0)= A_q$, and $\W(\u^*,0+;Y_0)=\W(\u^*,r;Y_0)=A_q$ for every $r\leq1$.
Therefore, $\u^*$ is self-similar in $Q_1^-$ with respect to the point $Y_0=(y_0,s_0)$.
Now apply again Theorem \ref{lowenergy} and Theorem \ref{lower-energy-for-q}  to find out $\u^*$ is a half-space solution with respect to $Y_0$, say $\u^*(x,t)=\alpha ((x^1-y_0^1)_+)^\kp$ for $t\leq s_0$. The uniqueness of solution of \eqref{system} (Lemma \ref{uniqueness-forward}) yields that the equality holds for $t\leq0$.
Notice that $|\u^*(0,0)|=0$ since $X_i\in\Gamma(\u)$. Therefore, $y_0^1=0$ and $\u^*(x,t)=\alpha (x^1_+)^\kp$, which contradicts \eqref{openness-rel2}.
\end{proof}


\section{H\"older regularity of  $\partial_t \u$}\label{sec:Holder}

Our way of approach, as mentioned in the introduction, is to use elliptic regularity theory for the free boundary problems. This approach is based on using the epiperimetric inequality for the elliptic systems as done in \cite{asuw15, FSW20}. The reduction of parabolic problem to the elliptic case was successfully used in \cite{W2000}. 
The idea is that near  regular points of the free boundary, where the blow-up regime is half-space,  the time derivative of the solutions vanishes faster than the order of scaling which is $\kp=2/(1-q)$.  This enables us to apply the epiperimetric inequality.

Our strategy is to  prove that $\partial_t\u$ is subcaloric and vanishes continuously on the free boundary (when $q=0$).  So we can deduce the H\"older regularity for it. This method needs a modification for $q>0$.
We start by following lemma which is essential in the case $q>0$.

\begin{lemma}\label{cauchy-inverse-estimate}
Let $(x_0,t_0)\in \Gamma^\kp$ be a regular  free boundary point of $\emph{\u}$. Then for every $\epsilon>0$, there exists $r_0>0$ such that 
\begin{equation}\label{lem5.1-rel1}
 |\emph{\u}|^2|\nabla\emph{\u}|^2\leq (1+\epsilon)\left|\sum_{k=1}^m u^k\nabla  u^k\right|^2,\quad\text{in } Q_{r_0}(x_0,t_0).
 \end{equation} 
\end{lemma}
\begin{proof}
By contradiction consider the sequence  $(x_j,t_j)\rightarrow(x_0,t_0)$ at which inequality \eqref{lem5.1-rel1} does not hold. 
Let $d_j:=\sup \{r:Q_r^-(x_j,t_j)\subset \{|\u|>0\} \}$ and $(y_j,s_j)\in \partial_pQ_{d_j}^-(x_j,t_j) \cap \Gamma(\u)$. 
According to the openness of regular points, see Proposition \ref{R-points}, we imply that $(y_j,s_j)$ are regular points of free boundary. 
Now, employing the growth estimates of solutions near $\Gamma^\kp$, Theorem \ref{thm:growth} (as well as Remark \ref{forward-regularity}), and possibly passing to a subsequence, we may assume that 
\begin{equation*}
\frac{\u(d_j x +x_j, d_j^2 t+t_j)}{d_j^\kp}:=\u_j(x,t)\rightarrow \u_0(x,t),
\end{equation*}
and 
\begin{equation*}
((y_j-x_j)/d_j, (s_j-t_j)/d_j^2):=(\tilde{y}_j, \tilde{s}_j)\rightarrow (\tilde{y}_0, \tilde{s}_0)\in \partial_p Q_1^-.
\end{equation*}
Therefore, inequality \eqref{lem5.1-rel1} can not be true for $\u_j$ at point $(0,0)$. We will show that $\u_0$ is a half-space solution with respect to $(\tilde{y}_0, \tilde{s}_0)$, i.e.  
\begin{equation}\label{lem5.1-rel2}
\u_0(x,t)=\alpha((x-\tilde{y}_0)\cdot \nu)_+^\kp\e, \textrm{ in } \R^n\times (-\infty, \tilde{s}_0],
\end{equation}
for some $\nu\in\R^n$, $\e\in\R^m$. By the uniqueness of forward problem, Lemma \ref{uniqueness-forward}, the representation \eqref{lem5.1-rel2} is valid for $t\in(-\infty,0]$ and $\u_0$ must satisfy the equality  $ |\u_0|^2|\nabla\u_0|^2=\left|\sum_{k=1}^m u_0^k\nabla  u_0^k\right|^2$. The contradiction proves the lemma.

In order to show that $\u_0$ is a half-space solution,  let $ \varrho>0$ be a small number, 
such that $ Q_\varrho (x_0,t_0) \cap \Gamma$ consists only of regular points; see Proposition \ref{R-points}.
For every $ r>0$ denote by 
\begin{equation*}
w_r(X):=\W(\u \eta,r; X).
\end{equation*}
Then $w_r$ is continuous, and has a pointwise limit,  as $r \rightarrow 0$. 
Since $ Q_\varrho (x_0,t_0) \cap \Gamma$ consists only of regular points, then 
\begin{equation}\label{lem5.1-wrconv}
\lim_{r\rightarrow 0+} w_r(x,t) =A_q, ~\textrm{ for } (x,t) \in Q_\varrho (x_0,t_0) \cap \Gamma.
\end{equation}
 Furthermore, by  monotonicity formula  $ w_r(x,t)+F(r)$ is a nondecreasing function in $r$,  hence by Dini's monotone convergence theorem, the convergence in \eqref{lem5.1-wrconv}
is uniform. Thus
\begin{equation*}
 \W( \u_0, r; \tilde{y}_0, \tilde{s}_0)=\lim_{j \rightarrow \infty} \W( \u_j \eta, r; \tilde{y}_j, \tilde{s}_j)
  = \lim_{j \rightarrow \infty} \W(\u \eta, d_jr; y_j, {s}_j)=A_q,
 \end{equation*}
 for any $r>0$.  
Hence $\u_0$ is backward self-similar with respect to $ ( \tilde{y}_0, \tilde{s}_0)$.
To finish the argument, note that $(\tilde y_j,\tilde s_j)$ is a regular point of $\u_j$ and consider the convergence $\u_j\rightarrow\u_0$ in $Q_2^-$. Then the interior of $\{\u_0=0\}$ is not empty and by Theorem \ref{lower-energy-for-q} we infer that $\u_0$ must be a half-space solution with respect to the point $(\tilde{y}_0, \tilde{s}_0)$.
\end{proof}

\begin{lemma} \label{dtsubcaloric}
Let 
\begin{equation*}
g(x,t):=|\partial_t \emph{\u} (x,t)|^2|\emph{\u} (x,t)|^{-2q},
\end{equation*}
\begin{enumerate}[(i)]
\item If $q=0$, then $ g(x,t)=|\partial_t \emph{\u} (x,t)|^2$ is a subcaloric function in the set $\{ |\emph{\u}| >0\}$.

\item If $0<q<1$ and $(x_0,t_0)\in \Gamma^\kp$ is a regular free boundary point. Then there exists $0<r_0$ and $\theta\geq 2$ such that $g^{\theta} $ is a subcaloric function in the set $\{ |\emph{\u}| >0\}\cap Q_{r_0}(x_0,t_0)$.
\end{enumerate}
\end{lemma}

\begin{proof}
$(i)$ By direct calculations;
\begin{equation*}
\Delta |\partial_t \u |^2 =\sum_{k=1}^m \Delta (\partial_t u^k)^2=2\sum_{k=1}^m \partial_t u^k \Delta \partial_t u^k+2 \sum_{k=1}^m |\nabla \partial_t u^k |^2
\end{equation*}
and 
\begin{equation*}
\partial_t  |\partial_t \u |^2=2 \sum_{k=1}^m \partial_t u^k \partial^2_{tt}u^k.
\end{equation*}
Hence calculating and using  the Cauchy-Schwarz inequality, we obtain
\begin{equation*}
\begin{aligned}
H( |\partial_t \u |^2) =&2\sum_{k=1}^m\partial_t  u^kH( \partial_t u^k)  +2 \sum_{k=1}^m |\nabla \partial_t u^k|^2=
2\sum_{k=1}^m \partial_t u^k \frac{\partial}{\partial t}\left(\frac{u^k}{|\u|‌^{1-q}}\right) +2 \sum_{k=1}^m |\nabla \partial_t u^k|^2 \\
=&2\sum_{k=1}^m \left(  \frac{(\partial_t u^k)^2}{| \u |^{1-q}}- (1-q)\frac{u^k \partial_t u^k\sum_{j=1}^m u^j \partial_t u^j}{|\u|^{3-q}} \right)+2 \sum_{k=1}^m |\nabla \partial_t u^k |^2 \\
=&\frac{2}{|\u|^{3-q}} \left( |\u |^2 |\partial_t \u|^2 -(1-q)(\u \cdot \partial_t \u )^2\right)+ 2 \sum_{k=1}^m |\nabla \partial_t u^k|^2 \geq 0.
\end{aligned}
\end{equation*}

$(ii)$ Since 
$H(g^\theta)=\theta g^{\theta-2}(gHg+(\theta-1)|\nabla g|^2)$, 
it is enough to show that 
\begin{equation}\label{5.2-eq1}
gHg+(\theta-1)|\nabla g|^2\geq 0.
\end{equation}
Note that this relation is valid for $\theta\geq2$ in $\{ |\u| >0\}$ regardless of whether    $\partial_t\u$ vanishes or not. 
We can write
\begin{align*}
Hg=|\u|^{-2q}H(|\partial_t\u|^2)+|\partial_t\u|^2 H(|\u|^{-2q})+2\nabla(|\partial_t\u|^2)\cdot\nabla(|\u|^{-2q}).
\end{align*}
From part $(i)$, we know that 
\begin{equation*}
H(|\partial_t\u|^2)=\frac{2}{|\u|^{3-q}} \left( |\u |^2 |\partial_t \u|^2 -(1-q)(\u \cdot \partial_t \u )^2\right)+ 2 \sum_{k=1}^m |\nabla \partial_t u^k|^2,
\end{equation*}
and by a direct calculation we obtain,
\begin{align*}
H(|\u|^{-2q})=-2q|\u|^{-q-1}-2q|\u|^{-2q-2}|\nabla\u|^2+4q(1+q)|\u|^{-2q-4} \left|\sum_{k=1}^m u^k\nabla  u^k\right|^2.
\end{align*}
Then
\begin{align*}
\frac12 gHg=&(1-q)|\partial_t \u|^2|\u|^{-3q-3}\left( |\u |^2 |\partial_t \u|^2 -(\u \cdot \partial_t \u )^2\right)
-q|\partial_t \u|^4|\u|^{-4q-2}|\nabla\u|^2\\
&+|\partial_t \u|^2|\u|^{-4q}|\nabla\partial_t\u|^2
+2q(1+q)|\partial_t \u|^4|\u|^{-4q-4}\left|\sum_{k=1}^m u^k\nabla  u^k\right|^2\\
&-4q|\partial_t \u|^2|\u|^{-4q-2}\left(\sum_{k=1}^m \partial_tu^k\nabla  \partial_tu^k\right)\cdot\left(\sum_{k=1}^m u^k\nabla  u^k\right).
\end{align*}
According to Lemma \ref{cauchy-inverse-estimate}, we can assume that 
$$ |\u|^2|\nabla\u|^2\leq (1+\epsilon)\left|\sum_{k=1}^m u^k\nabla  u^k\right|^2,$$ 
in a neighborhood of $(x_0,t_0)$ for some $\epsilon>0$ which is determined later.
Therefore, in order to prove \eqref{5.2-eq1} in this neighborhood we have 
\begin{align*}
\frac12 gHg\geq & |\u|^{-4q}\left|\sum_{k=1}^m \partial_tu^k\nabla  \partial_tu^k\right|^2
+q(2q+1-\epsilon)|\partial_t \u|^4|\u|^{-4q-4}\left|\sum_{k=1}^m u^k\nabla  u^k\right|^2\\
&-4q|\partial_t \u|^2|\u|^{-4q-2}\left(\sum_{k=1}^m \partial_tu^k\nabla  \partial_tu^k\right)\cdot\left(\sum_{k=1}^m u^k\nabla  u^k\right)\\
\geq & -2(\theta-1)\left| |\u|^{-2q}\left(\sum_{k=1}^m \partial_tu^k\nabla  \partial_tu^k\right)-q|\partial_t \u|^2 |\u|^{-2q-2} \left(\sum_{k=1}^m u^k\nabla  u^k\right)\right|^2 =- \frac12(\theta-1)|\nabla g|^2
\end{align*}
where the last inequality holds when $2\theta q\leq (2\theta-1)(1-\epsilon)$. We can choose suitable $\epsilon>0$ provided $2\theta>\frac1{1-q}$.
\end{proof}

Now we  prove that the time derivative vanishes continuously 
on the regular part of the  free boundary.

\begin{lemma}\label{dtcontinuous}
Let $g$ be the function defined in Lemma \ref{dtsubcaloric} and suppose $ (x_0,t_0) \in \Gamma^\kp$ is a regular free boundary point, then 
\begin{equation*}
\lim_{(x,t)\rightarrow (x_0,t_0)}g(x,t)=0.
\end{equation*}
\end{lemma}

\begin{proof}   
Let $ (x_j,t_j)\rightarrow (x_0,t_0) $ be a maximizing sequence in the sense that 
\begin{equation*}
\lim_{j\rightarrow +\infty} g(x_j,t_j)= \limsup_{(x,t)\rightarrow(x_0,t_0)}g (x,t):=m^2>0.
\end{equation*}
Let $d_j:=\sup \{r:Q_r^-(x_j,t_j)\subset \{|\u|>0\} \}$ and $  (y_j,s_j)\in \partial_pQ_{d_j}^-(x_j,t_j) \cap \Gamma$. 
Following  the same lines of proof as that of Lemma \ref{cauchy-inverse-estimate}, we may assume 
\begin{equation*}
\frac{\u(d_j x +x_j, d_j^2t+t_j)}{d_j^\kp}:=\u_j(x,t)\rightarrow \u_0(x,t),
\end{equation*}
\begin{equation*}
((y_j-x_j)/d_j, (s_j-t_j)/d_j^2):=(\tilde{y}_j, \tilde{s}_j)\rightarrow (\tilde{y}_0, \tilde{s}_0)\in \partial_p Q_1^-,
\end{equation*}
and
\begin{equation}\label{lem5-3-rel1}
\u_0(x,t)=\alpha((x-\tilde{y}_0)\cdot \nu)_+^\kp\e, \textrm{ in } \R^n\times (-\infty, 0].
\end{equation}
Since $Q_1^-\subset \{ |\u_j|>0\}$, then $ Q_1^-\subset \{ |\u_0|>0\}$, and the convergence is uniform in $Q_1^-$.
Hence  
$$| \partial_t \u_0(0,0)|\,|\u_0(0,0)|^{-q} =\lim_{j \to \infty} |\partial_t \u_j(0,0)|\,|\u_j(0,0)|^{-q} =\lim_{j \to \infty} |\partial_t \u( x_j, t_j)|\,|\u( x_j, t_j)|^{-q}  = m, $$
and for all $
(x,t)\in Q_1^-$,
\begin{align*}
 |\partial_t \u_0(x,t)| \,|\u_0(x,t)|^{-q}& =\lim_{j \to \infty} |\partial_t \u_j(x,t)| \,|\u_j(x,t)|^{-q} \\
 &=\lim_{j \to \infty} |\partial_t \u( d_j x+x_j, d_j^2 t +t_j)| \,|\u(d_j x+x_j, d_j^2 t +t_j)|^{-q}  \leq m. 
 \end{align*}

 Since $|\partial_t \u|^2$ is subcaloric for $q=0$ or  $g^\theta$ for $q>0$
 (Lemma \ref{dtsubcaloric}), we can 
 apply the maximum principle to arrive at $ |\partial_t \u_0(x,t)|=m|\u_0(x,t)|^q$ in the connected component of $Q_1^-$, containing the origin, 
which contradicts  \eqref{lem5-3-rel1}.
\end{proof}

Now using a standard iterative argument one can prove the H\"older regularity 
of the time derivative. 

\begin{lemma}\label{holder}
Let $g$ be the function defined in Lemma \ref{dtsubcaloric} and suppose $ (x_0,t_0) \in \Gamma $ is a regular free boundary point.
Then $g$ is  a H\"{o}lder continuous function in a neighbourhood of $(x_0,t_0)$.
\end{lemma}

\begin{proof}
Lemma \ref{dtsubcaloric} and Lemma \ref{dtcontinuous} together imply that $g$ (or $g^\theta$ for $0<q<1$) is a continuous subcaloric function in a neighbourhood of regular points (we extend $g$ to zero in $\{\u=0\}$).
Since the coincidence set $\{ \u = 0 \}$  close to regular points are uniformly large, 
 we may invoke  Lemma A4 in \cite{Caf77}, which states that 
  if $ h \leq M$ in $Q_1:= B_1 \times (0,1)$ is a continuous subcaloric function and 
\begin{equation*}
\frac{|Q_1 \cap \{h <M/2\}|}{|Q_1|}>\lambda>0, 
\end{equation*}
then there exists $0<\gamma=\gamma(\lambda)<1$ such that 
\begin{equation*}
h(0,1)<\gamma M.
\end{equation*}
Since $g$  is subcaloric, we obtain that
\begin{equation*}
\sup_{Q_{r/2}} g(x,t) \leq \gamma \sup_{Q_{r}} g(x,t).
\end{equation*}
Fix $(x,t) \in Q_{{r_0}/2}(x_0,t_0)$, then there exists $k\geq 1$ such that 
$2^{-k-1}r_0< |x|\leq 2^{-k}r_0$, and 
\begin{equation*}
g(x,t)   \leq \gamma \sup_{Q_{2^{-k+1}r_0}} g(x,t)  \leq \gamma^k \sup_{Q_{r_0}} g(x,t)  \leq  \frac{1}{\gamma}\left( \frac{|x|}{r_0}\right)^{-\frac{\ln \gamma}{\ln 2}}\sup_{Q_{r_0}} g(x,t).
\end{equation*}
Hence the function $ g$ is H\"{o}lder continuous with the exponent $ \beta = -\frac{\ln \gamma}{2 \ln 2}$.
\end{proof}

\begin{corollary}\label{holder-time-derivative}
Let  $\emph{\u}$ be a   solution to \eqref{system} and suppose  that $(x_0,t_0)\in\Gamma^\kp(\emph{\u})$ is a regular point,  then there exists constants $C$, $0<r_0<1$ and $0<\beta<1$ such that
$$
\sup_{Q^-_r(x_0,t_0)} |\partial_t\emph{\u}| \leq C r^{\kp-2+\beta}, \qquad  \forall \ 0<r<r_0.
$$
\end{corollary}


\section{Regularity of the free boundary}\label{sec:regularity}

We consider the following local (fixed time) version of balanced  energy;

\begin{align*} 
{W}_{t_0}(\u,r,x_0):=\frac{1}{r^{n+2\kp-2} } \int_{B_r(x_0)}     | \nabla \u(x, t_0) |^2& + \frac2{1+q} {|\u(x,t_0)|^{1+q}} dx\\
& - \frac{\kp}{r^{n+2\kp-1}} \int_{\partial B_r(x_0)} |\u(x,t_0)|^2 d \mathcal{H}^{n-1}.
\end{align*}
\begin{proposition} \label{Wprop}
Let $(x_0,t_0)\in \Gamma^\kp$ be a regular free boundary point, then
there exist constants $C>0$ and $0<\beta<1$, such that
\begin{equation*}
\left |{W}_{t_0}(\emph{\u},r_2,x_0)-{W}_{t_0}(\emph{\u},r_1,x_0)-2 \int_{r_1}^{r_2}   r \int_{\partial B_1(0)} \left| \frac{d }{dr}  \emph{\u}_{r, t_0}\right|^2 d \mathcal{H}^{n-1}dr   \right| \leq C|r_2^\beta-r_1^\beta|.
\end{equation*}
\end{proposition}

\begin{proof}
Let us denote  by $ \u_{r,t_0}:=\frac{\u(rx+x_0,t_0)}{r^\kp}$, then
\begin{equation*} 
{W}_{t_0}(\u,r,x_0)= \int_{B_1(0)}     | \nabla \u_{r, t_0}|^2 +\frac2{1+q} |\u_{r,t_0}|^{1+q} dx - \kp
\int_{\partial B_1(0)} |\u_{r,t_0}|^2 d \mathcal{H}^{n-1}.
\end{equation*}
Hence
\begin{equation} \label{derloc}
\begin{aligned}
\frac{d}{dr}{W}_{t_0}(\u,r,x_0)&= 2 \int_{B_1(0)}      \nabla \u_{r, t_0}   \nabla \frac{d}{dr} \u_{r, t_0}+ 
\frac{\u_{r,t_0}}{ |\u_{r,t_0}|^{1-q}}\frac{d }{dr}  \u_{r, t_0}dx 
- 2\kp\int_{\partial B_1(0)} \u_{r,t_0} \frac{d }{dr}  \u_{r, t_0} d \mathcal{H}^{n-1} \\
&=2 \int_{B_1(0)} { - \partial_t \u(rx+x_0, t_0) } \frac{d}{dr} \u_{r, t_0}dx
+2 r \int_{\partial B_1(0)} \left| \frac{d }{dr}  \u_{r, t_0}\right|^2 d \mathcal{H}^{n-1}.
\end{aligned}
\end{equation}
Letting $ r_2>r_1>0$ and integrating  \eqref{derloc} in the interval $(r_1,r_2)$, we obtain
\begin{equation*} 
\begin{aligned}
\Big |{W}_{t_0}(\u,r_2,x_0)-{W}_{t_0}(\u,r_1,x_0)-&2 \int_{r_1}^{r_2}   r \int_{\partial B_1(0)} \left| \frac{d }{dr}  \u_{r, t_0}\right|^2 d \mathcal{H}^{n-1}dr  \Big|\\
\leq &2\int_{r_1}^{r_2} \int_{B_1(0)}{ | \partial_t \u(rx+x_0, t_0)| }\left|  \frac{d}{dr} \u_{r, t_0}\right|dxdr\\
\leq &C \int_{r_1}^{r_2} \int_{B_1(0)} \frac{ | \partial_t \u(rx+x_0, t_0)| }{r}dxdr\\
\leq &C_1  \int_{r_1}^{r_2} r^{\beta-1} dr = \frac{C_1}{\beta}(r_2^\beta-r_1^\beta).
\end{aligned}
\end{equation*} 
\end{proof}

The following epiperimetric inequality from \cite{asuw15} and \cite{FSW20} will be used to treat the parabolic system.

\begin{theorem}[Epiperimetric inequality]
There exists $ \varepsilon \in (0,1)$ and $ \delta >0$ such that if $ \emph{\c}=\emph{\c}(x)$ is a backward self-similar function of degree $\kp$ satisfying 
\begin{equation*}
|| \emph{\c} -\emph{\h} ||_{W^{1,2}(B_1, \R^m)}+ || \emph{\c} -\emph{\h} ||_{L^\infty (B_1, \R^m)} \leq \delta, \textrm{ for some } \h \in \bH,
\end{equation*}
then there exists $ \emph{\v} \in W^{1,2}(B_1; \R^m)$ such that $\emph{\v} =\emph{\c} $ on $\partial B_1$ and
 \begin{equation*}
M(\emph{\v})-M(\emph{\h}) \leq (1-\varepsilon)\left( M(\emph{\c})- M(\emph{\h})\right),
\end{equation*}
where
\begin{equation*}
M(\emph{\v}):=\int_{B_1}(|\nabla \emph{\v}|^2+\frac2{1+q}|\emph{\v}|^{1+q})dx -\kp\int_{\partial B_1} |\emph{\v}|^2d \mathcal{H}^{n-1}.
\end{equation*}
\end{theorem}

\begin{proposition} (Energy decay, uniqueness of blow-up limits)
Let $ (x_0,t_0)\in \Gamma$ be a regular point, and $ \emph{\u}_0$ be any blow-up of $ \emph{\u}$ at $(x_0,t_0)$.
Suppose that the epiperimetric inequality holds with $0<\varepsilon <1$ for each 
\begin{equation*}
\emph{\c}_r(x, t_0):=|x|^\kp \emph{\u}_{r,t_0}\left(x/|x|,t_0\right)=\frac{|x|^\kp}{r^\kp} \emph{\u}\left(x_0+rx/|x|,t_0\right)
\end{equation*}
and for all $r \leq r_0$.
Then there exists 
$C>0$ and $ 0<\gamma<1$ such that 
\begin{equation}\label{gamma}
\left| {W}_{t_0}(\emph{\u},r,x_0)-{W}_{t_0}(\emph{\u},0+,x_0) \right| \leq Cr^\gamma, \textrm{ for small } r>0,
\end{equation}
and
\begin{equation}\label{gammahalf}
\int_{\partial B_1(0)} \left| \frac{\emph{\u}(rx+x_0,t_0)}{r^\kp}-\emph{\u}_0(x)\right| d \mathcal{H}^{n-1}\leq C r^{\gamma/2}, 
\end{equation}
therefore $\emph{\u}_0$ is the unique blow-up limit of $\emph{\u}$ at the point $(x_0,t_0)$.
\end{proposition}

\begin{proof}
Let $ e(r):= {W}_{t_0}(\u,r,x_0)-{W}_{t_0}(\u,0+,x_0) $,  then 
\begin{equation*}
\begin{aligned}
e^{\prime}(r)=&-\frac{n+2\kp-2}{r^{n+2\kp-1} } \int_{B_r(x_0)}     | \nabla \u(x, t_0) |^2 + \frac2{1+q} {|\u(x,t_0)|^{1+q}} dx +\frac{\kp(n+2\kp-1)}{r^{n+2\kp}} \int_{\partial B_r(x_0)} |\u(x,t_0)|^2 d \mathcal{H}^{n-1} \\
&+ \frac{1}{r^{n+2\kp-2} } \int_{\partial B_r(x_0)}     | \nabla \u(x, t_0) |^2 + 2 {|\u(x,t_0)|} d\mathcal{H}^{n-1} \\
&-\frac{\kp}{r^{n+2\kp-1}}\left( 
2\int_{\partial B_r(x_0)} (\nabla \u(x,t_0)  \cdot \nu   ) \cdot \u(x,t_0) d \mathcal{H}^{n-1} 
+\frac{n-1}{r}\int_{\partial B_r(x_0)} |\u(x,t_0)|^2 d \mathcal{H}^{n-1} \right)\\
=&-\frac{n+2\kp-2}{r}\left(e(r)+{W}_{t_0}(\u,0+,x_0)\right)
 +\frac{1}{r } \int_{\partial B_1(x_0)}     | \nabla \u_r(x, t_0) |^2 + \frac2{1+q} {|\u_r(x,t_0)|^{1+q}} d\mathcal{H}^{n-1}\\
&-\frac{2\kp}{r}\int_{\partial B_1(x_0)}( \nabla \u_r(x,t_0) \cdot  \nu  )\cdot \u_r(x,t_0) d \mathcal{H}^{n-1} 
-\frac{\kp(n-2)}{r}\int_{\partial B_1(x_0)} |\u_r(x,t_0)|^2 d \mathcal{H}^{n-1}\\
\geq&-\frac{n+2\kp-2}{r}\left(e(r)+{W}_{t_0}(\u,0+)\right)
+
\frac{1}{r}\int_{\partial B_1(0)} \left(  |\nabla_\theta \u_r|^2+\frac2{1+q}|\u_r|^{1+q}-(\kp(n-2)+\kp^2) |\u_r|^2 \right)d \mathcal{H}^{n-1}\\
=&
-\frac{n+2\kp-2}{r}\left(e(r)+{W}_{t_0}(\u,0+)\right)
+
\frac{1}{r}\int_{\partial B_1(0)} \left(  |\nabla_\theta \c_r|^2+\frac2{1+q}|\c_r|^{1+q}-(\kp(n-2)+\kp^2) |\c_r|^2 \right)d \mathcal{H}^{n-1}\\
=&\frac{n+2\kp-2}{r}\left(M(\c_r)-e(r)-{W}_{t_0}(\u,0+,x_0)\right)\\
 \geq&
\frac{n+2\kp-2}{r} \left(\frac{M(\v) -{W}_{t_0}(\u,0+,x_0)}{1-\varepsilon}-e(r)   \right),
\end{aligned}
\end{equation*}
where we  employed the epiperimetric inequality in the last step. Now let us observe that
 $\u_{r,t_0}$ minimises the following energy
 \begin{equation*}
 J(\v):=\int_{B_1(0)}  |\nabla \v|^2 +\frac2{1+q}|\v|^{1+q} dx +2r^{2-\kp} \int_{B_1(0)} \v \cdot \partial_t \u(x_0+rx,t_0)dx,
 \end{equation*}
 where $ \v =\u_{r,t_0}$ on $\partial B_1(0)$.
Hence
\begin{equation*}
\begin{aligned}
M(\v)=&\int_{B_1} |\nabla \v|^2+\frac2{1+q}| \v|^{1+q}dx -\kp\int_{\partial B_1} |\v|^2d\mathcal{H}^{n-1}\\
=&J(\v)-2r^{2-\kp} \int_{B_1} \v \cdot \partial_t \u(x_0+rx,t_0)dx
-\kp\int_{\partial B_1} |\u_{r,t_0}|^2d\mathcal{H}^{n-1} \\
\geq& M(\u_{r,t_0})+2r^{2-\kp}  \int_{B_1}(\u_{r,t_0}- \v) \cdot \partial_t \u(x_0+rx,t_0)dx.
\end{aligned}
\end{equation*}
Now we may conclude that
\begin{equation}
\begin{aligned}
\label{edecay}
e^\prime(r) \geq &
\frac{n+2\kp-2}{r(1-\varepsilon)} \left({M(\u_{r,t_0}) -{W}_{t_0}(\u,0+,x_0)} +2r^{2-\kp} \int_{B_1}(\u_{r,t_0}- \v) \cdot \partial_t \u(x_0+rx,t_0)dx \right)\\
&- \frac{(n+2\kp-2)e(r)}{r}
=\frac{(n+2\kp-2)e(r)}{r(1-\varepsilon)} - \frac{(n+2\kp-2)e(r)}{r}\\
&+2  \frac{(n+2\kp-2)}{(1-\varepsilon)} r^{1-\kp} \int_{B_1} (\u_{r,t_0}- \v)\cdot \partial_t \u(x_0+rx,t_0)dx\\
\geq &\frac{\varepsilon (n+2\kp-2)e(r)}{r(1-\varepsilon)}-Cr^{\beta-1},
\end{aligned}
\end{equation}
by Corollary \ref{holder-time-derivative}. 
It follows from \eqref{edecay} that
\begin{equation*}
\frac{d}{dr}\left( e(r) r^{-\frac{\varepsilon(n+2\kp-2)}{1-\varepsilon}}\right)\geq -Cr^{\beta-1-\frac{\varepsilon(n+2\kp-2)}{1-\varepsilon}}.
\end{equation*}
Integrating the last inequality from $r$ to $1$, we obtain 
\begin{equation*}
e(1)- e(r) r^{-\frac{\varepsilon(n+2\kp-2)}{1-\varepsilon}} \geq -\frac{C}{\beta-\frac{\varepsilon(n+2\kp-2)}{1-\varepsilon}}\left( 1- r^{\beta-\frac{\varepsilon(n+2\kp-2)}{1-\varepsilon}}\right),
\end{equation*}
and therefore
\begin{equation*}
e(r) \leq e(1) r^{\frac{\varepsilon(n+2\kp-2)}{1-\varepsilon}}   -\frac{C}{\beta-\frac{\varepsilon(n+2\kp-2)}{1-\varepsilon}}\left( r^{\frac{\varepsilon(n+2\kp-2)}{1-\varepsilon}}- r^{\beta}\right)
\leq C_0 r^\gamma,
\end{equation*}
where $\gamma:=\min \left(\beta, \frac{\varepsilon(n+2\kp-2)}{1-\varepsilon} \right)$, and $C_0>0$  depends only on the given parameters.
The proof of \eqref{gamma} is now complete, and we proceed to the proof of \eqref{gammahalf}. \\

\par
Let $2^{-l} <r_1\leq 2^{-l+1}\leq 2^{-k}<r_2 <2^{-k+1}$, where $ k,l \in \mathbb{N}$.
It is easy to see that 
\begin{equation*}
\begin{aligned}
\int_{\partial B_1(0)}& \left|\frac{\u(x_0+r_1x,t_0)}{r_1^\kp}  -\frac{\u(x_0+r_2x,t_0)}{r_2^\kp}    \right|  d\mathcal{H}^{n-1} \leq 
\int_{\partial B_1(0)}  \int_{r_1}^{r_2} \left| \frac{d }{dr}  \u_{r,t_0}\right| dr d\mathcal{H}^{n-1}  \\
\leq & \sum_{j=l}^k \int_{\partial B_1(0)}  \int_{2^{-j}}^{2^{-j+1}} \left|\frac{d }{dr}  \u_{r,t_0}  \right| dr d\mathcal{H}^{n-1} \\
\leq  &C_n \sum_{j=l}^k \left( \int_{\partial B_1(0)}  \int_{2^{-j}}^{2^{-j+1}} r \left|\frac{d }{dr}  \u_{r,t_0}  \right|^2  dr d\mathcal{H}^{n-1} \right)^\frac{1}{2}
\end{aligned}
\end{equation*}
By Proposition \ref{Wprop} and relation \eqref{gamma}, we can estimate
\begin{equation*}
\int_{\partial B_1(0)}  \int_{2^{-j}}^{2^{-j+1}} r \left|\frac{d }{dr}  \u_{r,t_0}  \right|^2  dr d\mathcal{H}^{n-1} \leq 
C (2^{(-j+1)\gamma}- 2^{-j\gamma}) \leq C2^{-\gamma j}.
\end{equation*}
Hence
\begin{equation*}
\begin{aligned}
\int_{\partial B_1(0)} &\left|\frac{\u(x_0+r_1x,t_0)}{r_1^\kp}  -\frac{\u(x_0+r_2x,t_0)}{r_2^\kp}    \right|  d\mathcal{H}^{n-1} \leq 
C   \sum_{j=l}^k 2^{-\gamma j/2}\\
=&C\frac{2^{-\gamma l/2}-2^{-\gamma(k+1)/2}}{1-2^{-\gamma/2}} \leq  \frac{C}{2^{\gamma/2}-1}(r_2^{\gamma/2}-r_1^{\gamma/2}),
\end{aligned}
\end{equation*}
and \eqref{gammahalf} follows.
\end{proof}

The following theorem has been proved as Theorem 4.7 in \cite{FSW20} (for $q>0$) and Theorem 4 in \cite{asuw15} (for $q=0$).

\begin{theorem}\label{converge-speed}
Let $C_h$ be a compact set of points $x_0\in\Gamma^\kp_{t_0}$ with the following property:
at least one blow-up limit $\emph{\u}_0$ of $\emph{\u}(rx+x_0,t_0)/r^\kp$   is a half-space solution, say $\emph{\u}_0(x)=\alpha\max(x\cdot\nu(x_0,t_0),0)^\kp\emph{\e}(x_0,t_0)$ for some $\nu(x_0,t_0)\in\partial B_1\subset\R^n$ and $\emph{\e}(x_0,t_0)\in\partial B_1\subset\R^m$. Then there exist $r_0$ and $C<\infty$ such that
$$
\int_{\partial B_1} \left| \frac{\emph{\u}(rx+x_0,t_0)}{r^\kp}-\alpha\max(x\cdot\nu(x_0,t_0),0)^\kp\emph{\e}(x_0,t_0)\right| d \mathcal{H}^{n-1}\leq C r^{\gamma/2},
$$
for every $x_0\in C_h$ and every $r\leq r_0$.
\end{theorem}

\begin{theorem}\label{thm:FBregularity}
In a neighbourhood of regular points the free boundary is $C^{1,\alpha}$ in space and  $C^{0,1/2}$ in time.
\end{theorem}

\begin{proof}
First, consider the normal vectors $\nu(x_0,t_0)$ and $\e(x_0,t_0)$ defined in Theorem \ref{converge-speed}, we show that $(x_0,t_0)\mapsto\nu(x_0,t_0)$ and $(x_0,t_0)\mapsto\e(x_0,t_0)$ are H\"older continuous with exponent $\beta=\frac\gamma{\gamma+2\kp}$.

Therefore,  it follows that for each time section the free boundary 
is $C^{1,\beta} $, provided the free boundary point is a regular point.
This in turn implies that the free boundary is a graph in the time direction, close to such points. 
To see that the free boundary is half-Lipschitz in time, we may perform a blow-up at free boundaries, 
along with a contradiction argument. This is standard and left to the reader.
\end{proof}


\section{Appendix}
\begin{lemma}[Uniqueness of forward problem]\label{uniqueness-forward}
Let $\emph{\u}$ and $\emph{\v}$ be global solutions of \eqref{system} in $\R^n\times(-\infty,t_0]$ which have polynomial growth. If $\emph{\u}(\cdot,s)=\emph{\v}(\cdot,s)$ for some $s<t_0$,
 then $\emph{\u}(\cdot,t)=\emph{\v}(\cdot,t)$ for all $s\leq t\leq t_0$.
\end{lemma}
\begin{proof}
Multiply $H(\u-\v)=f(\u)-f(\v)$ by $(\u-\v)G$ and integrate 
\begin{align*}
0\leq\int_s^\tau\int_{\R^n}(f(\u)-f(\v))\cdot(\u-\v)Gdxdt=-\int_s^\tau\int_{\R^n}&\Big[\frac12\partial_t|\u-\v|^2+|\nabla(\u-\v)|^2\\
&+\left(\nabla(\u-\v)\cdot\frac x{2t}\right)\cdot(\u-\v)\Big]Gdxdt.
\end{align*}
Let $\w:=\u-\v$, then
 \begin{align*}
 \frac12\int_{\R^n}|\w(x,\tau)|^2G(x,\tau)dx\leq &\int_s^\tau\int_{\R^n} \frac12|\w|^2\partial_tG-\left[|\nabla\w|^2+\left(\nabla\w\cdot\frac x{2t}\right)\cdot\w\right]Gdxdt\\
\leq  &\int_s^\tau\int_{\R^n} \left[-\frac{|x|^2+2nt}{8t^2}|\w|^2-|\nabla\w|^2+|\nabla\w|\,|\frac x{2t}|\,|\w|\right]Gdxdt\\
\leq & \int_s^\tau\int_{\R^n} \frac{n}{-4t}|\w|^2Gdxdt=:\phi(\tau).
\end{align*} 
Therefore, $-\frac{2\tau}n\phi'(\tau)\leq \phi(\tau)$ and so $\frac{d}{d\tau}\left[(-\tau)^{n/2}\phi(\tau)\right]\leq0$ for $s<\tau<0$. From $\phi(s)=0$, we conclude that $\phi(\tau)\equiv0$.
\end{proof}

\begin{lemma}\label{caloric-estimate-lemma}
Let $h$ be a caloric function in $\R^n\times(-4,0]$. Then for $s<t\leq0$ we have the following estimate
$$
e^{\frac{|x|^2}{t+s}}|h(x,t)|^2\leq  \left(\frac{\sqrt{3}s}{s-t}\right)^{n}\int_{\R^n}|h(y,s)|^2G(y,s)dy.
$$
\end{lemma}
\begin{proof}
By the representation of the caloric function, we have
\begin{align*}
h(x,t)=\int_{\R^n}h(y,s)G(x-y,s-t)dy.
\end{align*} 
Then
\begin{align}\label{caloric-estimate-lemma-eq:1}
|h(x,t)|^2\leq&\left(\int_{\R^n}|h(y,s)|^2G(y,s)dy\right)\left(\int_{\R^n}\frac{\left(G(x-y,s-t)\right)^2}{G(y,s)}dy\right)
\end{align} 
On the other hand, we can write
\begin{align*}
\frac{\left(G(x-y,s-t)\right)^2}{G(y,s)}=&\left(\frac{-s}{4\pi(t-s)^2}\right)^{n/2}\exp\left(-\frac{|x-y|^2}{2(t-s)}-\frac{|y|^2}{4s}\right)\\
\leq&\left(\frac{-s}{4\pi(t-s)^2}\right)^{n/2}\exp\left(-\frac{|x|^2}{2(t-s)}+\frac{x\cdot y}{t-s}-\frac{(t+s)|y|^2}{4s(t-s)}\right)\\
\leq&\left(\frac{-s}{4\pi(t-s)^2}\right)^{n/2}\exp\left(\left(-\frac12+\epsilon\right)\frac{|x|^2}{t-s}+\left(\frac1{4\epsilon}-\frac{t+s}{4s}\right)\frac{|y|^2}{t-s}\right).
\end{align*} 
For every $\epsilon>\frac{s}{t+s}$, we obtain that
\begin{align*}
\exp\left(\left(\frac12-\epsilon\right)\frac{|x|^2}{t-s}\right)\int_{\R^n}\frac{\left(G(x-y,s-t)\right)^2}{G(y,s)}dy\leq&
\left(\frac{-s}{4\pi(t-s)^2}\right)^{n/2}\int_{\R^n}\exp\left(\left(\frac1{4\epsilon}-\frac{t+s}{4s}\right)\frac{|y|^2}{t-s}\right)dy\\
=&\left(\frac{-s}{(t-s)\left(\frac{t+s}{s}-\frac1{\epsilon}\right)}\right)^{n/2}.
\end{align*} 
Now let $\epsilon=\frac{3s-t}{2(t+s)}$, so by \eqref{caloric-estimate-lemma-eq:1} the proof will be done.
\end{proof}

\begin{lemma}\label{appx-lemma2}
Assume that $\emph{\w}\in L^2(Q_4)$ has polynomial growth and $t<0$ fixed, then
$$
\int_{\R^n}|\emph{\w}(x,t)|^2\frac{|x|^2}{-t}G(x,t)dx\leq 4\int_{\R^n}\left(n|\emph{\w}(x,t)|^2-4t|\nabla \emph{\w}(x,t)|^2\right)G(x,t)dx.
$$
\end{lemma}
\begin{proof}
Using the relation $\nabla G(x,t)=\frac{x}{2t}G(x,t)$ to obtain
\begin{align*}
\int_{\R^n}|\w(x,t)|^2\frac{|x|^2}{-t}G(x,t)dx&=-2\int_{\R^n}|\w|^2(x\cdot\nabla G)dx=2\int_{\R^n}{\rm div}(|\w|^2x)Gdx\\
&=2n\int_{\R^n}|\w|^2Gdx+4\int_{\R^n}\w\cdot(\nabla\w\cdot x)Gdx\\
&\leq2n\int_{\R^n}|\w|^2Gdx+\int_{\R^n}|\w|^2\frac{|x|^2}{-2t}Gdx+\int_{\R^n}(-8t)|\nabla \w|^2Gdx.
\end{align*}
Now we can easily prove the lemma.
\end{proof}

\begin{lemma}\label{appex-lemma3}
Let $\emph{\u}$ be a function defined   in  $\R^n \times  [-R, 0)$ (for some $a, R >0$) with polynomial growth,  and $\p$ be a $\kp$-backward self-similar caloric vector-function.
Then  for $-R <  t_1 < t_2 < 0$
\begin{equation*}
\begin{aligned}
\int_{t_1}^{t_2} \int_{\R^n} \left( | \nabla (\emph{\p}-\emph{\u})|^2+\frac{\kp|\emph{\p}-\emph{\u}|^2}{2t} \right)G dx dt =
\int_{t_1}^{t_2}  \int_{\R^n} \left( | \nabla\emph{\u} |^2+\frac{\kp|\emph{\u}|^2}{2t} \right)G dxdt.
\end{aligned}
\end{equation*}
\end{lemma}
\begin{proof}
Since  $ \nabla G(x,t)= \frac{x}{2t} G(x,t)$, we have 
\begin{equation*}
(\nabla \u : \nabla \v) G= \nabla \u : \nabla (\v G) -(\nabla \u\cdot x)\cdot\frac{\v G}{2t}.
\end{equation*}
Obviously, $   |\nabla (\p-\u) |^2 = | \nabla \u |^2 - 2 \nabla \p : \nabla \u +| \nabla \p |^2$, hence
\begin{equation*}
\begin{aligned}
\int_{t_1}^{t_2}  \int_{\R^n} \left( | \nabla (\p-\u)|^2+\frac{\kp|\p-\u|^2}{2t} \right)G dx dt=&
\int_{t_1}^{t_2}  \int_{\R^n} \left( | \nabla \u |^2+\frac{\kp|\u|^2}{2t} \right)G dx dt \\
&+\int_{t_1}^{t_2}  \int_{\R^n}  \nabla \p : (\nabla \p-2 \nabla \u ) G dx dt\\
&+ \int_{t_1}^{t_2}  \int_{\R^n}  \kp\frac{\p\cdot(\p-2\u)}{2t} G dxdt\\
=& \int_{t_1}^{t_2} \int_{\R^n} \left( | \nabla \u |^2+\frac{\kp|\u|^2}{2t} \right)G dx dt \\
&-\int_{t_1}^{t_2}  \int_{\R^n} \left( \Delta \p+ \frac{1}{2t} x \cdot \nabla \p -\frac{\kp\p}{2t}  \right)\cdot (\p-2\u)G dx dt\\
=& \int_{t_1}^{t_2}  \int_{\R^n} \left( | \nabla \u |^2+\frac{\kp|\u|^2}{2t} \right)G dx dt
\end{aligned}
\end{equation*}
where we used integration by parts and that 
\begin{equation*}
\Delta \p +\frac{1}{2t} x \cdot \nabla \p -\frac{\kp\p}{2t}=\partial_t \p +\frac{1}{2t} x \cdot \nabla \p -\frac\kp2\frac{\p}{t} = \frac{1}{2t}L \p=0.
\qedhere
\end{equation*}
\end{proof}

 The following lemma is an extension of Lemma 1.1 in \cite{caffarelli1985partial} to the parabolic case.

\begin{lemma}\label{holder-regularity}
Consider $\beta>0$ to be noninteger, and let $u(x,t)$ be a function satisfying 
$$
|H(u)(X)|\leq C_*|X|^\beta.
$$
Then there is a caloric polynomial $P$ of degree at most $\lfloor\beta\rfloor+2$ such that
$$
\lVert u-P \rVert_{L^\infty(Q_r^-)}\leq CC_*r^{\beta+2},\qquad \text{for }r\in(0,1),
$$
where constant $C$ depends only on $n, \beta$ and $\lVert u\rVert_{L^\infty(Q_1^-)}$.
\end{lemma}
\begin{proof}
We can assume that $\lVert u\rVert_{L^\infty(Q_1^-)}\leq 1$ and $C_*\leq \delta$, where $\delta$ is small enough and will be determined later. (Replace $u$ by $u(R^{-1}x,R^{-2}t)$ for a large fixed constant $R$ to find $|H(u)|\leq \delta |X|^\beta$)
The proof of lemma is based on the following claim.\\

{\bf Claim:} There exists $0<\rho<1$ and a sequence of caloric polynomials $P_k$ such that
$$
\lVert u-P_k \rVert_{L^\infty(Q_{\rho^k}^-)}\leq \rho^{k(\beta+2)},
$$
and
$$
|\partial_x^\mu\partial_t^\ell(P_{k}-P_{k-1})(0,0)|\leq C_0\rho^{(k-1)(\beta+2-|\mu|-2\ell)}, \quad\text{if }|\mu|+2\ell<\beta+2.
$$
\\

A straight forward implication of this claim is that the sequence $\{P_k\}$ converges uniformly in $Q_1$ to a polynomial $P$ of degree at most $\lfloor\beta\rfloor+2$ which clearly satisfies
\begin{align*}
\lVert u-P \rVert_{L^\infty(Q_{\rho^k}^-)}\leq& \lVert u-P_k \rVert_{L^\infty(Q_{\rho^k}^-)}+\sum_{i=k+1}^\infty\lVert P_i-P_{i-1} \rVert_{L^\infty(Q_{\rho^k}^-)}\\
\leq&\rho^{k(\beta+2)}+\sum_{i=k+1}^\infty \sum_{|\mu|+2\ell<\beta+2}C_0\rho^{(i-1)(\beta+2-|\mu|-2\ell)}\rho^{k(|\mu|+2\ell)}\\
\leq&\rho^{k(\beta+2)}+\sum_{|\mu|+2\ell<\beta+2}C_0\rho^{k(\beta+2)}\leq C_{n,\beta} C_0\rho^{k(\beta+2)}.
\end{align*}
Therefore, the lemma will be proved for $C:=\frac1\delta C_{n,\beta} C_0\rho^{-(\beta+2)}$.
\\

Now we prove the claim. It is obviously true for $k=0$ (just take $P_0\equiv P_{-1}\equiv0$). We now assume that it holds for $k$ and we prove it for $k+1$. Define 
$$
v(X):=\frac{u(\rho^kx,\rho^{2k}t)-P_k(\rho^kx,\rho^{2k}t)}{\rho^{k(\beta+2)}}. 
$$
Then by inductive hypothesis $|v|\leq1$ in $Q_1^-$. In addition, 
$$
|H(v)|=\left|\frac{H(u)(\rho^kx,\rho^{2k}t)}{\rho^{k\beta}}\right|\leq C_*\leq\delta.
$$
If we apply Lemma 6.1 in \cite{figalli2015general}, there exist $\delta=\delta(\epsilon)$ and function $w$ satisfying 
$$
|v-w|\leq \epsilon,\quad\text{ in }Q_{1/2}^-,
$$
and
$$
\left\{\begin{array}{ll}H(w)=0&\text{in }Q_{1/2}^-,\\[8pt]
w=v&\text{on }\partial_pQ_{1/2}^-.\end{array}\right.
$$
Now consider a polynomial $\hat P$  of degree at most $\lfloor\beta\rfloor+2$ such that
$\partial_x^\mu\partial_t^\ell\hat P(0,0)=\partial_x^\mu\partial_t^\ell w(0,0)$ for $|\mu|+\ell<\beta+2$.
Since $\lVert w\rVert_{L^\infty(Q_{1/2}^-)}\leq\lVert v\rVert_{L^\infty(Q_{1}^-)}\leq1$, by estimates on derivatives for caloric functions $|\partial_x^\mu\partial_t^\ell\hat P(0,0)|\leq C_0$ for a universal constant $C_0$.
Obviously,  $\hat P$ is caloric and 
$$
\lVert w-\hat P\rVert_{L^\infty(Q_{\rho}^-)}\leq C_0\rho^{\lfloor\beta\rfloor+3}.
$$ 
In particular, if we choose $\rho$ sufficiently small so that $C_0\rho^{\lfloor\beta\rfloor+3}\leq \frac12\rho^{\beta+2}$ and then choose $\epsilon$ such that $\epsilon\leq\frac12\rho^{\beta+2}$, we arrive at
$$
\lVert v-\hat P\rVert_{L^\infty(Q_{\rho}^-)}\leq \rho^{\beta+2},
$$
or equivalently 
$$
\lVert u-P_{k+1} \rVert_{L^\infty(Q_{\rho^{k+1}}^-)}\leq \rho^{(k+1)(\beta+2)},\quad
P_{k+1}(X):=P_k(X)+\rho^{k(\beta+2)}\hat P(\rho^{-k}x,\rho^{-2k}t).
$$
We also have 
$$
|\partial_x^\mu\partial_t^\ell(P_{k+1}-P_{k})(0,0)|\leq \rho^{k(\beta+2-|\mu|-\ell)}|\partial_x^\mu\partial_t^\ell\hat P(0,0)|\leq C_0\rho^{k(\beta+2-|\mu|-\ell)}.
$$
\end{proof}

\paragraph{\bf{Acknowledgements.} }
H. Shahgholian was partially supported by Swedish Research Council.
G. Aleksanyan thanks KTH for visiting appointment. 

\addcontentsline{toc}{section}{\numberline{}References}

\end{document}